\pdfoutput=1
\documentclass{amsart}%
\usepackage{amsfonts}
\usepackage{amsmath}
\usepackage{amssymb}
\usepackage{graphicx}%
\providecommand{\U}[1]{\protect\rule{.1in}{.1in}}
\newtheorem{theorem}{Theorem}
\theoremstyle{plain}

\newtheorem{corollary}{Corollary}

\newtheorem{definition}{Definition}

\newtheorem{lemma}{Lemma}

\newtheorem{proposition}{Proposition}
\newtheorem{remark}{Remark}

\numberwithin{equation}{section}
\begin{document}
\title[Projective IFS]{Real projective iterated function systems}
\author{Michael F. Barnsley}
\address{Department of Mathematics\\
Australian National University\\
Canberra, ACT, Australia}
\author{Andrew Vince}
\address{Department of Mathematics\\
University of Florida\\
Gainesville, FL 32611-8105, USA}
\email{avince@ufl.edu}
\urladdr{http://www.superfractals.com{}}

\begin{abstract}
This paper contains four main results associated with an attractor of a
projective iterated function system (IFS). The first theorem characterizes
when a projective IFS has an attractor which avoids a hyperplane. The second
theorem establishes that a projective IFS has at most one attractor. In the
third theorem the classical duality between points and hyperplanes in
projective space leads to connections between attractors that avoid
hyperplanes and repellers that avoid points, as well as hyperplane attractors
that avoid points and repellers that avoid hyperplanes. Finally, an index is
defined for attractors which avoid a hyperplane. This index is shown to be a
nontrivial projective invariant.

\end{abstract}
\maketitle

\section{Introduction}

This paper provides the foundations of a surprisingly rich mathematical theory
associated with the attractor of a real projective iterated function system
(IFS). (A real projective IFS consists of a finite set of projective
transformations $\{f_{m}:\mathcal{P\rightarrow P}\}_{m=1}^{M}$ where
$\mathcal{P}$ is a real projective space. An attractor is a nonempty compact
set $A\subset\mathcal{P}$ such that $\lim_{k\rightarrow\infty}\mathcal{F}%
^{k}(B)=\mathcal{F}\left(  A\right)  =A$ for all nonempty sets $B$ in an open
neighborhood of $A$, where $\mathcal{F}(B)=\cup_{m=1}^{M}f_{m}(B)$.) In
addition to proving conditions which guarantee the existence and uniqueness of
an attractor for a projective IFS, we also present several related concepts.
The first connects an attractor which avoids a hyperplane with a hyperplane
repeller. The second uses information about the hyperplane repeller to define
a new index for an attractor. This index is both invariant under projective
transformations and nontrivial, which implies that it joins the cross ratio
and Hausdorff dimension as nontrivial invariants under the projective group.
Thus, these attractors belong in a natural way to the collection of
geometrical objects of classical projective geometry.

The definitions that support expressions such as "iterated function system",
"attractor", "basin of attraction" and "avoids a hyperplane", used in this
Introduction, are given in Section \ref{hilbert}.

Iterated function systems are a standard framework for describing and
analyzing self-referential sets such as deterministic fractals
\cite{barnsleydemko, FE1, hutchinson} and some types of random fractals
\cite{barnsleyhutch}. Attractors of affine IFSs have many applications,
including image compression \cite{barnsleyhurd, barnsleynotices, fisher} and
geometric modeling \cite{blanc}. They relate to the theory of the joint
spectral radius \cite{berger} and to wavelets \cite{Berger2}. Projective IFSs
have more degrees of freedom than comparable affine IFSs \cite{superfractals}
while the constituent functions share geometrical properties such as
preservation of straight lines and cross ratios. Projective IFSs have been
used in digital imaging and computer graphics, see for example \cite{notices2}%
, and they may have applications to image compression, as proposed in \cite[p.
10]{IMA}. Projective IFSs can be designed so that their attractors are smooth
objects such as arcs of circles and parabolas, and rough objects such as
fractal interpolation functions.

The behavior of attractors of projective IFSs appears to be complicated. In
computer experiments conducted by the authors, attractors seem to come and go
in a mysterious manner as parameters of the IFS are changed continuously. See
Example 4 in Section \ref{examplesec} for an example that illustrates such
phenomena. The intuition developed for affine IFSs regarding the control of
attractors seems to be wrong in the projective setting. Our theorems provide
insight into such behavior.

One key issue is the relationship between the existence of an attractor and
the contractive properties of the functions of the IFS. In a previous paper
\cite{ABVW} we investigated the relationship between the existence of
attractors and the existence of contractive metrics for IFSs consisting of
affine maps on ${\mathbb{R}}^{n}$. We established that an affine IFS
$\mathcal{F}$ has an attractor if and only if $\mathcal{F}$ is contractive on
all of ${\mathbb{R}}^{n}$. In the present paper we focus on the setting where
$\mathbb{X}=\mathbb{P}^{n}$ is real $n$-dimensional projective space and each
function in $\mathcal{F}$ is a projective transformations. In this case
$\mathcal{F}$ is called a \textit{projective IFS}.

Our first main result, Theorem \ref{main}, provides a set of equivalent
characterizations of a projective IFS that possesses an attractor that avoids
a hyperplane. The adjoint $\mathcal{F}^{t}$ of a projective IFS $\mathcal{F}$
is defined in Section \ref{dual}, and convex body is defined in Definition
\ref{convbdy}. An IFS $\mathcal{F}$ is contractive on $S\subset X$ when
$\mathcal{F}(S)\subset S$ and there is a metric on $S$ with respect to which
all the functions of the IFS are contractive, see Definition
\ref{contract-def}. For a set $X$ in a topological space, $\overline{X}$
denotes its closure, and $int(X)$ denotes its interior.

\begin{theorem}
\label{main} If $\mathcal{F}$ is a projective IFS on $\mathbb{P}^{n}$, then
the following statements are equivalent.

\begin{enumerate}
\item $\mathcal{F}$ has an attractor $A$ that avoids a hyperplane.

\item There is a nonempty open set $U$ that avoids a hyperplane such that
$\mathcal{F}(\overline{U})\subset U$.

\item There is a nonempty finite collection of disjoint convex bodies
$\left\{  C_{i}\right\}  $ such that $\mathcal{F}(\cup_{i}C_{i})\subset
int(\cup_{i}C_{i})$.

\item There is a nonempty open set $U\subset\mathbb{P}^{n}$ such that
$\mathcal{F}$ is contractive on $\overline{U}$.

\item The adjoint projective IFS $\mathcal{F}^{t}$ has an attractor $A^{t}$
that avoids a hyperplane.
\end{enumerate}

When these statements are true we say that $\mathcal{F}$ is contractive.
\end{theorem}

Statement (4) is of particular importance because if an IFS is contractive,
then it possesses an attractor that depends continuously on the functions of
the IFS, see for example \cite[Section 3.11]{FE1}. Moreover, if an IFS is
contractive, then various canonical measures, supported on its attractor, can
be computed by means of the "chaos game" algorithm \cite{barnsleydemko}, and
diverse applications, such as those mentioned above, become feasible. Note
that statement (4) of Theorem~\ref{main} immediately implies uniqueness of an
attractor in the set $U$, but not uniqueness in $\mathbb{P}^{n}$. See also
Remark \ref{remark1} in Section \ref{remarksec}.

Our second main result establishes uniqueness of attractors, independently of
whether or not Theorem 1 applies.

\begin{theorem}
\label{uniquethm} A projective IFS has at most one attractor.
\end{theorem}

The classical projective duality between points and hyperplanes manifests
itself in interesting ways in the theory of projective IFSs. Theorem
\ref{adjointhm} below, which depends on statement (5) in Theorem \ref{main},
is an example. It is a geometrical description of the dynamics of
$\mathcal{F}$ as a set operator on $\mathbb{P}^{n}.$ The terminology used is
provided in Section \ref{dual}.

\begin{theorem}
\label{adjointhm}

(1) A projective IFS has an attractor that avoids a hyperplane if and only if
it has a hyperplane repeller that avoids a point. The basin of attraction of
the attractor is the complement of the union of the hyperplanes in the repeller.

(2) A projective IFS has a hyperplane attractor that avoids a point if and
only if it has a repeller that avoids a hyperplane. The basin of attraction of
the hyperplane attractor is the set of hyperplanes that do not intersect the repeller.
\end{theorem}

Figure \ref{diskleaf1} illustates Theorem \ref{adjointhm}. Here and in the
other figures we use the disk model of the projective plane. Diametrically
opposite points on the boundary of the disk are identified in $\mathbb{P}^{2}%
$. In the left-hand panel of Figure \ref{diskleaf1} the "leaf" is the
attractor $A$ of a certain projective IFS $\mathcal{F}$ consisting of four
projective transformations on $\mathbb{P}^{2}$. The surrounding grainy region
approximates the set $R$ of points in the corresponding hyperplane repeller.
The complement of $R$ is the basin of attraction of $A$. The central green,
red, and yellow objects in the right panel comprise the attractor of the
adjoint IFS $\mathcal{F}^{t}$, while the grainy orange scimitar-shaped region
illustrates the corresponding hyperplane repeller.%
\begin{figure}[ptb]%
\centering
\includegraphics[
natheight=13.652900in,
natwidth=27.466900in,
height=2.2133in,
width=4.4234in
]%
{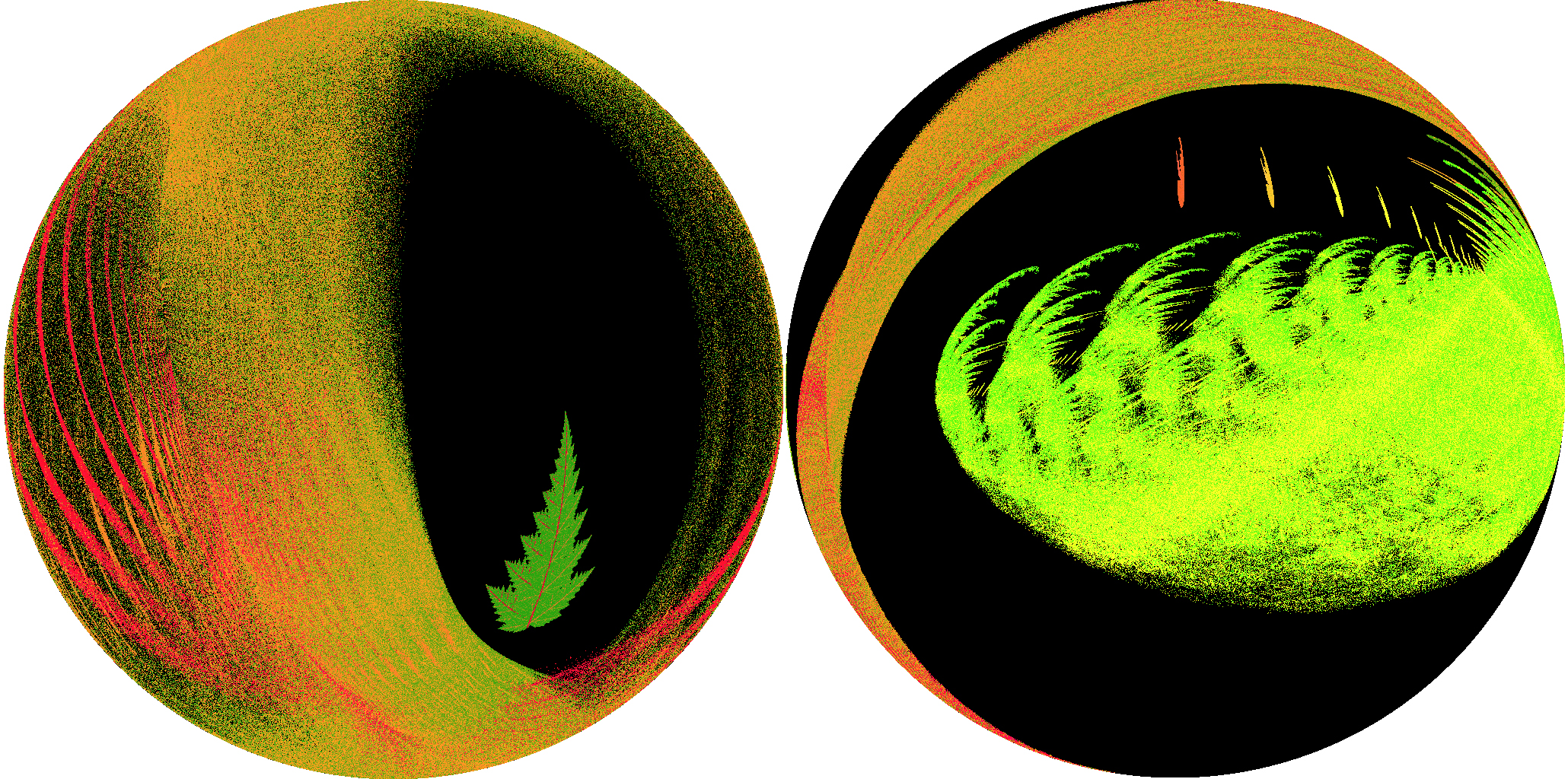}%
\caption{The image on the left shows the attractor and hyperplane repeller of
a projective IFS. The basin of attraction of the leaf-like attractor is the
black convex region together with the leaf. The image on the right shows the
attractor and repeller of the adjoint system.}%
\label{diskleaf1}%
\end{figure}

Theorem \ref{adjointhm} enables us to associate a geometrical index with an
attractor that avoids a hyperplane. More specifically, if an attractor $A$
avoids a hyperplane then $A$ lies in the complement of (the union of the
hyperplanes in) the repeller. Since the connected components of this
complement form an open cover of $A$ and since $A$ is compact, $A$ is actually
contained in a finite set of components of the complement. These observations
lead to the definition of a geometric index of $A$, $index(A),$ as is made
precise in Definition \ref{indexdef2}. This index is an integer associated
with an attractor $A$, not any particular IFS that generates $A$. As shown in
Section~\ref{groupinvsec}, as a consequence of Theorem 4\textbf{,} this index
is nontrivial, in the sense that it can take positive integer values other
than one. Moreover, it is invariant under under $PGL(n+1,\mathbb{R}),$ the
group of real, dimension $n$, projective transformations. That is,
$index(A)=index(g(A))$ for all $g\in PGL(n+1,\mathbb{R})$.

See Remark \ref{remark3} of Section \ref{remarksec} concerning attractors and
repellers in the case of affine IFSs. See Remark \ref{remark4} in
Section~\ref{remarksec} concerning the fact that the Hausdorff dimension of
the attractor is also an invariant under the projective group.

\section{\label{orgsec}Organization}

Since the proofs of our results are quite complicated, this section describes
the structure of this paper, including an overview of the proof of Theorem
\ref{main}.

Section~\ref{hilbert} contains definitions and notation related to iterated
function systems, and background information on projective space, convex sets
in projective space, and the Hilbert metric.

Section \ref{examplesec} provides examples that illustrate the intricacy of
projective IFSs and the value of our results. These examples also illustrate
the role of the avoided hyperplane in statements (1), (2) and (5) of Theorem
\ref{main}.

The proof of Theorem \ref{main} is achieved by showing that
\[
(1)\Rightarrow(2)\Rightarrow(3)\Rightarrow(4)\Rightarrow(1)\Leftrightarrow
(5).
\]

Section \ref{1implies2sec} contains the proof that $(1)\Rightarrow(2),$ by
means of a topological argument. Statement $(2)$ states that the IFS
$\mathcal{F}$ is a \textquotedblleft topological contraction\textquotedblright%
\ in the sense that it sends a nonempty compact set into its interior.

Section \ref{proofoflemmaD} contains the proof of Proposition \ref{lemmaD},
which describes the action of a projective transformation on the convex hull
of a connected set in terms of its action on the connected set. This is a key
result that is used subsequently.

Section~\ref{2implies3sec} contains the proof that $(2)\Rightarrow(3)$ by
means of a geometrical argument, in Lemmas \ref{(2)implies(3)A} and
\ref{(2)implies(3)B}. Statement $(3)$ states that the compact set, in
statement $(2)$, that is sent into its interior can be chosen to be the
disjoint union of finitely many convex bodies. What makes the proof somewhat
subtle is that, in general, there is no single convex body that is mapped into
its interior.

Sections \ref{p1of3implies4sec} and \ref{p2of3implies4sec} contain the proof
that $(3)\Rightarrow(4)$. Statement $(4)$ states that, with respect to an
appropriate metric, each function in $\mathcal{F}$ is a contraction. The
requisite metric is constructed in two stages. On each of the convex bodies in
statement $(3)$, the metric is basically the Hilbert metric as discussed in
Section \ref{hilbert}. How to combine these metrics into a single metric on
the union of the convex bodies is what requires the two sections.

Section \ref{(4)implies(1)sec} contains both the proof that $(4)\Rightarrow
(1)$ and the proof of Theorem \ref{uniquethm}.

Section \ref{dual} contains the proof that $(1)\Leftrightarrow(5),$ namely
that $\mathcal{F}$ has an attractor if and only if $\mathcal{F}^{t}$ has an
attractor. The adjoint IFS $\mathcal{F}^{t}$ consists of those projective
transformations which, when expressed as matrices, are the transposes of the
matrices that represent the functions of $\mathcal{F}$. The proof relies on
properties of an operation, called the complementary dual, that takes subsets
of ${\mathbb{P}^{n}}$ to subsets of ${\mathbb{P}^{n}}$.

Section \ref{dual} also contains the proof of Theorem \ref{adjointhm}, which
concerns the relationship between attractors and repellers. The proof relies
on classical duality between ${\mathbb{P}^{n}}$ and its dual
$\widehat{{\mathbb{P}^{n}}}$, as well as equivalence of statement (4) in
Theorem \ref{main}. Note that, if $\mathcal{F}$ has an attractor $A$ then the
orbit under $\mathcal{F}$ of any compact set in the basin of attraction of $A$
will converge to $A$ in the Hausdorff metric. Theorem \ref{adjointhm} tells us
that if $A$ avoids a hyperplane, then there is also a set $\mathcal{R}$ of
hyperplanes that repel, under the action of $\mathcal{F}$, hyperplanes
\textquotedblleft close\textquotedblright\ to $\mathcal{R}$. The hyperplane
repeller $\mathcal{R}$ is such that the IFS $\mathcal{F}^{-1},$ consisting of
all inverses of functions in $\mathcal{F}$, when applied to the dual space of
${\mathbb{P}^{n}}$, has $\mathcal{R}$ as an attractor. The relationship
between the hyperplane repeller of an IFS $\mathcal{F}$ and the attractor of
the adjoint IFS $\mathcal{F}^{t}$ is described in Proposition
\ref{dualIFSlemma}.

Section \ref{groupinvsec} considers properties of attractors that are
invariant under the projective group $PGL(n+1,\mathbb{R})$ . In particular, we
define $index(A)$ of an attractor $A$ that avoids a hyperplane, and establish
Theorem \ref{indexthm} which shows that this index is a nontrivial group invariant.

Section \ref{remarksec} contains various remarks that add germane information
that could interrupt the flow on a first reading. In particular, the topic of
non-contractive projective IFSs that, nevertheless, have attractors is
mentioned. Other areas open to future research are also mentioned.

\section{\label{hilbert}Iterated Function Systems, Projective Space, Convex
Sets, and the Hilbert Metric}

\subsection{\label{IFSsec}Iterated Function Systems and their Attractors}

\begin{definition}
Let $\mathbb{X}$ be a complete metric space. If $f_{m}:\mathbb{X}%
\rightarrow\mathbb{X}$, $m=1,2,\dots,M,$ are continuous mappings, then
$\mathcal{F}=\left(  \mathbb{X};f_{1},f_{2},...,f_{M}\right)  $ is called an
\textbf{iterated function system} (IFS).
\end{definition}

To define the attractor of an IFS, first define
\[
\mathcal{F}(B)=\bigcup_{f\in\mathcal{F}}f(B)
\]
for any $B\subset\mathbb{X}$. By slight abuse of terminology we use the same
symbol $\mathcal{F}$ for the IFS, the set of functions in the IFS, and for the
above mapping. For $B\subset\mathbb{X}$, let $\mathcal{F}^{k}(B)$ denote the
$k$-fold composition of $\mathcal{F}$, the union of $f_{i_{1}}\circ f_{i_{2}%
}\circ\cdots\circ f_{i_{k}}(B)$ over all finite words $i_{1}i_{2}\cdots i_{k}$
of length $k.$ Define $\mathcal{F}^{0}(B)=B.$

\begin{definition}
\label{attractdef}A nonempty compact set $A\subset\mathbb{X}$ is said to be an
\textbf{attractor} of the IFS $\mathcal{F}$ if

(i) $\mathcal{F}(A)=A$ and

(ii) there is an open set $U\subset\mathbb{X}$ such that $A\subset U$ and
$\lim_{k\rightarrow\infty}\mathcal{F}^{k}(B)=A,$ for all compact sets
$B\subset U$, where the limit is with respect to the Hausdorff metric.

The largest open set $U$ such that (ii) is true is called the \textbf{basin of
attraction} [for the attractor $A$ of the IFS $\mathcal{F}$].
\end{definition}

See Remark \ref{attractremark} in Section \ref{remarksec} concerning a
different definition of attractor.

\begin{definition}
\label{contract-def}A function $f:\mathbb{X}\rightarrow\mathbb{X}$ is called a
\textit{contraction} with respect to a metric $d$ if there is $0\leq\alpha<1$
such that $d(f(x),f(y))\leq\alpha\,d(x,y)$ for all $x,y\in{\mathbb{R}}^{n}$.

An IFS $\mathcal{F}=\left(  \mathbb{X};f_{1},f_{2},...,f_{M}\right)  $ is said
to be \textbf{contractive on a set} $U\subset\mathbb{X}$ if $\mathcal{F}%
(U)\subset U$ and there is a metric $d\,:\,U\times U\rightarrow\lbrack
0,\infty)$, giving the same topology as on $U$, such that, for each
$f\in\mathcal{F}$ the restriction $f|_{U}$ of $f$ to $U$ is a contraction on
$U$ with respect to $d$.
\end{definition}

\subsection{\label{projsec}Projective Space}

Let $\mathbb{R}^{n+1}$ denote $(n+1)$-dimensional Euclidean space and let
$\mathbb{P}^{n}$ denote real \textit{projective space}. Specifically,
$\mathbb{P}^{n}$ is the quotient of $\mathbb{R}^{n+1}\setminus\{0\}$ by the
equivalence relation which identifies $(x_{0},\dots,x_{n})$ with $(\lambda
x_{0},\dots,\lambda x_{n})$ for all nonzero $\lambda\in\mathbb{R}$. Let
\[
\phi:\,{\mathbb{R}}^{n+1}\setminus\{0\}\rightarrow\mathbb{P}^{n}%
\]
denote the canonical quotient map. The set $(x_{0},\dots,x_{n})$ of
coordinates of some $x\in{\mathbb{R}}^{n+1}$ such that $\phi(x)=p$ is referred
to as \textit{homogeneous coordinates} of $p$. If $p,q\in{\mathbb{P}^{n}}$
have homogeneous coordinates $(p_{0},\dots,p_{n})$ and $(q_{0},\dots,q_{n})$,
respectively, and $%
{\textstyle\sum\limits_{i=0}^{n}}
p_{i}q_{i}=0$, then we say that $p$ and $q$ are \textit{orthogonal}, and write
$p\bot q$. A \textit{hyperplane} in ${\mathbb{P}}^{n}$ is a \textit{ }set the
form%
\[
H=H_{p}=\{q\in\mathbb{P}^{n}:p\bot q=0\}\subset\mathbb{P}^{n}\text{,}%
\]
for some $p\in\mathbb{P}^{n}$.

\begin{definition}
A set $X\subset\mathbb{P}^{n}$ is said to \textbf{avoid a hyperplane} if there
exists a hyperplane $H\subset\mathbb{P}^{n}$ such that $H\cap X=\emptyset$.
\end{definition}

We define the \textquotedblleft round\textquotedblright\ metric $d_{\mathbb{P}%
}$ on $\mathbb{P}^{n}$ as follows. Each point $p$ of $\mathbb{P}^{n}$ is
represented by a line in $\mathbb{R}^{n+1}$ through the origin, or by the two
points $a_{p}$ and $b_{p}$ where this line intersects the unit sphere centered
at the origin. Then, in the obvious notation, $d_{{\mathbb{P}}}(p,q)=\min
\left\{  \left\Vert a_{p}-a_{q}\right\Vert ,\left\Vert a_{p}-b_{q}\right\Vert
\right\}  $ where $\left\Vert x-y\right\Vert $ denotes the Euclidean distance
between $x$ and $y$ in $\mathbb{R}^{n+1}$. In terms of homogeneous
coordinates, the metric is given by
\[
d_{\mathbb{P}}(p,q)=\sqrt{2-2\,\frac{|\langle p,q\rangle|}{\Vert p\Vert\Vert
q\Vert}},
\]
where $\langle\cdot,\cdot\rangle$ is the usual Euclidean inner product. The
metric space $(\mathbb{P}^{n},d_{\mathbb{P}})$ is compact.

A \textit{projective transformation} $f$ is an element of PGL$(n+1,{\mathbb{R}
})$, the quotient of GL$(n+1,{\mathbb{R}})$ by the multiples of the identity
matrix. A mapping $f\,:\,${$\mathbb{P}^{n}$}$\rightarrow\mathbb{P}^{n}$ is
well defined by $f(\phi x)=\phi(L_{f}x)$, where $L_{f}:{\mathbb{R}}
^{n+1}\rightarrow{\mathbb{R}}^{n+1}$ is any matrix representing projective
transformation $f$. In other words, the following diagram commutes:
\[%
\begin{array}
[c]{ccc}
& L_{f} & \\
{\mathbb{R}}^{n+1} & \rightarrow & {\mathbb{R}}^{n+1}\\
\phi\downarrow &  & \downarrow\phi\\
\mathbb{P}^{n} & \rightarrow & \mathbb{P}^{n}.\\
& f &
\end{array}
\]
When no confusion arises we may designate an $n$-dimensional projective
transformation $f$ by a matrix $L_{f} \in GL(n+1,{\mathbb{R}})$ that
represents it. An IFS $\mathcal{F}=\left(  \mathbb{P}{^{n}};f_{1}%
,f_{2},...,f_{M}\right)  $ is called a \textit{projective IFS} if each
$f\in\mathcal{F}$ is a projective transformation on $\mathbb{P}{^{n}}$.

\subsection{\label{convsec}Convex subsets of $\mathbb{P}^{n}$.}

We now define the notions of convex set, convex body, and convex hull of a set
with respect to a hyperplane. In Proposition \ref{lemmaD} we state an
invariance property that plays a key role in the proof of Theorem 1.

If $H\subset{\mathbb{P}}^{n}$ is a hyperplane, then there is a unique
hyperplane $\overline{H}\in{\mathbb{R}}^{n+1}$ such that $\phi(\overline
{H})=H$. If $p\in{\mathbb{P}}^{n}\setminus H$, there is a unique
$1$-dimensional subspace $\overline{p}\in{\mathbb{R}}^{n+1}$ such that
$\phi(\overline{p})=p.$ Let $u$ be a unit vector orthogonal to $\overline{H}$
and $W=\{x:\langle x,u\rangle=1\}$ be the corresponding affine subspace of
$\mathbb{R}^{n+1}$. Define a mapping $\theta\,:\,{\mathbb{P}}^{n}\setminus
H\rightarrow W$ by letting $\theta(p)$ be the intersection of $\overline{p}$
with $W$. Now $\theta$ is a surjective mapping from ${\mathbb{P}}^{n}\setminus
H$ onto the $n$-dimensional affine space $W$ such that projective subspaces of
${\mathbb{P}}^{n}\setminus H$ go to affine subspaces of $W$. In light of the
above, it makes sense to consider ${\mathbb{P}}^{n}\setminus H$ as an affine space.

\begin{definition}
\label{convbdy} A set $S\subset\mathbb{P}^{n}\backslash H$ is said to be
\textbf{convex} \textbf{with respect to a hyperplane }$H$ if $S$ is a convex
subset of $\mathbb{P}^{n}\backslash H$, considered as an affine space as
described above. Equivalently, with notation as in the above paragraph, $S$ is
convex with respect to $H$ if $\theta(S)$ is a convex subset of $W$. A closed
set that is convex with respect to a hyperplane and has nonempty interior is
called a\textbf{ convex body}.
\end{definition}

It is important to distinguish this definition of "convex" from projective
convex, which is the term often used to describe a set $S\subset\mathbb{P}%
^{n}$ with the property that if $l$ is a line in $\mathbb{P}^{n}$ then $S\cap
l$ is connected. (See \cite{dekker, groot} for a discussion of related matters.)

\begin{definition}
Given a hyperplane $H\subset\mathbb{P}^{n}$ and two points $x,y\in
\mathbb{P}^{n}\setminus H$, the unique line $\overline{xy}$ through $x$ and
$y$ is divided into two closed line segments by $x$ and $y$. The one that does
does \textit{not} intersect $H$ will be called the \textbf{line segment with
respect to }$H$\textbf{ }and denoted\textbf{ }$\overline{xy}_{H}$.
\end{definition}

Note that $C$ is convex with respect to a hyperplane $H$ if and only
if\textbf{ }$\overline{xy}_{H}\subset C$ for all $x,y\in C$.

\begin{definition}
\label{convhulldef}Let $S\subset\mathbb{P}^{n}$ and let $H$ be a hyperplane
such that $S\cap H=\emptyset$. The \textbf{convex hull of }$S$\textbf{ with
respect to }$H$ is
\[
conv_{H}(S)=conv(S),
\]
where $conv(S)$ is the usual convex hull of $S$, treated as a subset of the
affine space $\mathbb{P}^{n}\backslash H$. Equivalently, with notation as
above, if $S^{\prime}=conv(\theta(S)),$ where $conv$ denotes the ordinary
convex hull in $W$, then $conv_{H}(S)=\phi(S^{\prime}).$
\end{definition}

We can also describe $conv_{H}(S)$ as the smallest convex subset of
$\mathbb{P}^{n}\backslash H$ that contains $S$, i.e., the intersection of all
convex sets of $\mathbb{P}^{n}\backslash H$ containing $S$. The key result
concerning convexity and projective transformations is Proposition
\ref{lemmaD} in Section \ref{proofoflemmaD}.

\subsection{The Hilbert metric}

In this section we define the Hilbert metric associated with a convex body.

Let $p,q\in\mathbb{P}^{n}$, with $p\neq q$ and with homogeneous coordinates
$p=(p_{0},\dots,p_{n})$ and $q=(q_{0},\dots,q_{n})$. Any point $r$ on the line
$\overline{pq}$ has homogeneous coordinates $r_{i}=\alpha_{1}\,p_{i}%
+\alpha_{2}q_{i},\;i=0,1,\dots,n$. The pair $(\alpha_{1},\alpha_{2})$ is
referred to as the \textit{homogeneous parameters} of $r$ with respect to $p$
and $q$. Since the homogeneous coordinates of $p$ and $q$ are determined only
up to a scalar multiple, the same is true of the homogeneous parameters
$(\alpha_{1},\alpha_{2})$.

Let $a=(\alpha_{1},\alpha_{2}),b=(\beta_{1},\beta_{2}),c=(\gamma_{1}%
,\gamma_{2}),d=(\delta_{1},\delta_{2})$ be any four points on such a line in
terms of homogeneous parameters. Their \textit{cross ratio} $R(a,b,c,d)$, in
terms of homogeneous parameters on the projective line, is defined to be
\begin{equation}
R(a,b,c,d)=\frac{%
\begin{vmatrix}
\gamma_{1} & \alpha_{1}\\
\gamma_{2} & \alpha_{2}%
\end{vmatrix}
}{%
\begin{vmatrix}
\gamma_{1} & \beta_{1}\\
\gamma_{2} & \beta_{2}%
\end{vmatrix}
}\div\frac{%
\begin{vmatrix}
\delta_{1} & \alpha_{1}\\
\delta_{2} & \alpha_{2}%
\end{vmatrix}
}{%
\begin{vmatrix}
\delta_{1} & \beta_{1}\\
\delta_{2} & \beta_{2}%
\end{vmatrix}
}.
\end{equation}
The key property of the cross ratio is that it is invariant under any
projective transformation and under any change of basis $\{p,q\}$ for the
line. If none of the four points is the first base point $p$, then the
homogeneous parameters of the points are $(\alpha,1),(\beta,1),(\gamma
,1),(\delta,1)$ and the cross ratio can be expressed as the ratio of (signed)
distances:
\[
R(a,b,c,d)=\frac{(\gamma-\alpha)(\delta-\beta)}{(\gamma-\beta)(\delta-\alpha
)}.
\]

\begin{definition}
Let $K\subset\mathbb{P}^{n}$ be a convex body. Let $H\subset\mathbb{P}^{n}$ be
a hyperplane such that $H\cap K=\varnothing$. Let $x$ and $y$ be distinct
points in $int(K).$ Let $a$ and $b$ be two distinct points in the boundary of
$K$ such that $\overline{xy}_{H}\subset\overline{ab}_{H},$ where the order of
the points along the line segment $\overline{ab}_{H}$ is $a,x,y,b$. The
\textbf{Hilbert metric} $d_{K}$ on $int(K)$ is defined by%
\[
d_{K}(x,y)=\log R(a,b,x,y)=\log\left(  \frac{|ay|\,|bx|}{|ax|\,|by|}\right)
.
\]
Here $|ay|=\left\Vert a^{\prime}-y^{\prime}\right\Vert $,$|bx|=\left\Vert
b^{\prime}-x^{\prime}\right\Vert ,|ax|=\left\Vert a^{\prime}-x^{\prime
}\right\Vert ,\,|by|=\left\Vert b^{\prime}-y^{\prime}\right\Vert $ denote
Euclidean distances associated with any set of collinear points $a^{\prime
},x^{\prime},y^{\prime},b^{\prime}\in\mathbb{R}^{n+1}$ such that
$\phi(a^{\prime})=a$, $\phi(x^{\prime})=x$, $\phi(y^{\prime})=y$, and
$\phi(b^{\prime})=b$.
\end{definition}

A basic property of the Hilbert metric is that it is a projective invariant.
See \cite[p.105]{buseman} for a more complete discussion of the properties of
this metric. See Remark \ref{remark4} in Section \ref{remarksec} concerning
the relationship between the metrics $d_{\mathbb{P}}$ and $d_{K}$ and its
relevance to the evaluation and projective invariance of the Hausdorff dimension.

\section{\label{examplesec}Examples}

\noindent\textbf{EXAMPLE 1 }[IFSs with one transformation]: Let $\mathcal{F}%
=(\mathbb{P}^{n};f)$ be a projective IFS with a single transformation. By
Theorem~\ref{main} such an IFS has an attractor if and only if any matrix
$L_{f}$ repesenting $f$ has a dominant eigenvalue. (The map $L_{f}$ has a real
eigenvalue $\lambda_{0}$ with corresponding eigenspace of dimension $1$, such
that $\lambda_{0}>|\lambda|$ for every other eigenvalue $\lambda$.) For such
an IFS the attractor is a single point whose homogeneous coordinates are the
coordinates of the eigenvector corresponding to $\lambda_{0}$. The hyperplane
repeller of $\mathcal{F}$ is the single hyperplane $\phi(E)$, where $E$ is the
span of the eigenspaces corresponding to all eigenvalues of $L_{f}$ except
$\lambda_{0}$. The attractor of the adjoint IFS is also a single point,
$\phi(E^{\bot})$, where $E^{\bot}$ is the unique line through the origin in
$\mathbb{R}^{n+1}$ perpendicular to the hyperplane $E$.\medskip

\noindent\textbf{EXAMPLE 2} [Convex hull caveat ]: In Theorem \ref{main} the
implication $(2)\Rightarrow(3)$ contains a subtle issue. It may seem, at first
sight, to be trivial because surely one could choose $C$ simply to be the
convex hull of $U$? The following example shows that this is not true.
Let\textbf{ }$\mathcal{F}=(\mathbb{P}^{1};f_{1},f_{2})$ where
\[
f_{1}=
\begin{pmatrix}
4 & 0\\
1 & 1
\end{pmatrix}
,\qquad f_{2}=
\begin{pmatrix}
-4 & 0\\
1 & 1
\end{pmatrix}
.
\]
In $\mathbb{P}^{1}$ a hyperplane is just a point. Let $H_{0}=%
\begin{pmatrix}
0\\
1
\end{pmatrix}
$ and $H_{\infty}=%
\begin{pmatrix}
1\\
0
\end{pmatrix}
$ be two hyperplanes and consider the four points $p=%
\begin{pmatrix}
-9\\
1
\end{pmatrix}
,$ $q=%
\begin{pmatrix}
-2\\
1
\end{pmatrix}
,$ $r=%
\begin{pmatrix}
2\\
1
\end{pmatrix}
,$ and $s=%
\begin{pmatrix}
9\\
1
\end{pmatrix}
$ in $\mathbb{P}^{1}.$ Let $C_{1}$ be the line segment $\overline{pq}_{H_{0}}$
and let $C_{2}=\overline{rs}_{H_{0}}.$ There are two possible convex hulls of
$C_{1}\cup C_{2}$, one with \ respect to the hyperplane $H_{0}$ for example
and the other with respect to $H_{\infty}$ for example. It is routine to check
that $\mathcal{F}\left(  \overline{C_{1}\cup C_{2}}\right)  \subset C_{1}\cup
C_{2}$ but $\mathcal{F}\left(  \overline{conv_{H}(C_{1}\cup C_{2})}\right)
\nsubseteq conv_{H}(C_{1}\cup C_{2})$, where $H$ is either $H_{0}$ or
$H_{\infty}$. Thus the situation is fundamentally different from the affine
case; see \cite{ABVW}.\medskip

\noindent\textbf{EXAMPLE 3} [A non-contractive IFS with an attractor
]:\textit{ }Theorem \ref{main} leaves open the possible existence of a
non-contractive IFS that, nevertheless, has an attractor. According to Theorem
\ref{main} such an attractor must have nonempty intersection with every
hyperplane. The following example shows that such an IFS does exist. Let
$\mathcal{F}=\left(  P^{2};f_{1},f_{2}\right)  $ where
\[
f_{1}=
\begin{pmatrix}
1 & 0 & 0\\
0 & 2 & 0\\
0 & 0 & 2
\end{pmatrix}
\qquad\text{and}\qquad f_{2}=
\begin{pmatrix}
1 & 0 & 0\\
0 & 2\cos\,\theta & -2\sin\,\theta\\
0 & 2\sin\,\theta & 2\cos\,\theta
\end{pmatrix}
,
\]
and $\theta/\pi$ is irrational. In terms of homogeneous coordinates $(x,y,z)$,
the attractor of $\mathcal{F}$ is the line $x=0$.

Another example is illustrated in Figure \ref{special}, where
\[
f_{1}=
\begin{pmatrix}
41 & -19 & 19\\
-19 & 41 & 19\\
19 & 19 & 41
\end{pmatrix}
\qquad\text{and}\qquad f_{2}=
\begin{pmatrix}
-10 & -1 & 19\\
-10 & 21 & 1\\
10 & 10 & 10
\end{pmatrix}
,
\]
Neither function $f_{1}$ nor $f_{2}$ has an attractor, but the IFS consisting
of both of them does. The union $A$ of the points in the red and green lines
is the attractor. Since any two lines in $\mathbb{P}^{2}$ have nonempty
intersection, the attractor $A$ has nonempty intersection with every
hyperplane. Consequently by Theorem \ref{main}, there exist no metric with
respect to which both functions are contractive. In the right panel a zoom is
shown which displays the fractal structure of the set of lines that comprise
the attractor. The color red is used to indicate the image of the attractor
under $f_{1}$, while green indicates its image under $f_{2}$.

\begin{figure}[ptb]%
\centering
\includegraphics[
natheight=13.652900in,
natwidth=27.466900in,
height=2.0772in,
width=4.1486in
]%
{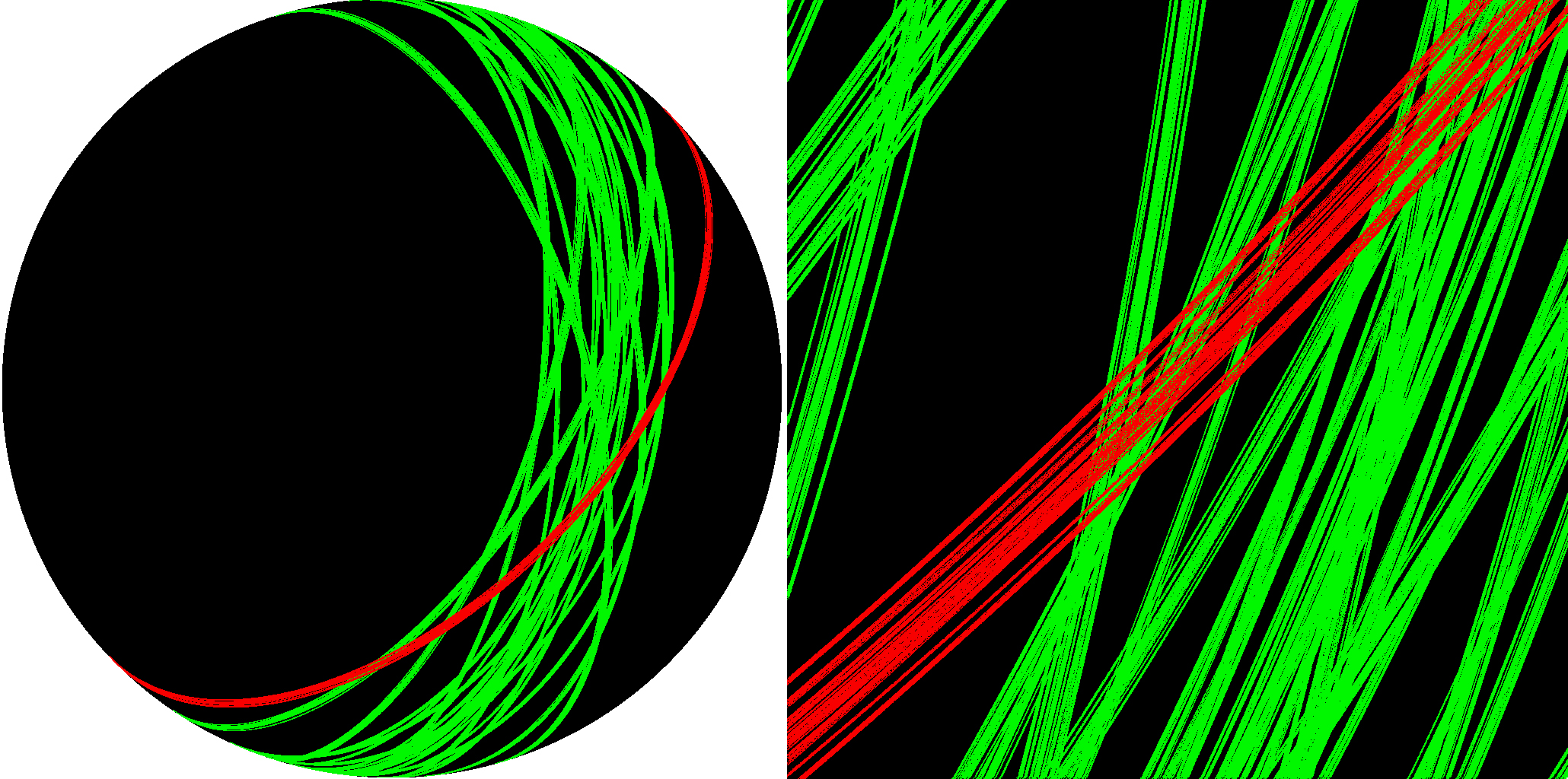}%
\caption{Projective attractor which includes a hyperplane, and a zoom. See
Example 3.}%
\label{special}%
\end{figure}
\medskip

\noindent\textbf{EXAMPLE 4} [Attractor discontinuity ]: This example consists
of a family $F=\{\mathcal{F}(t):t\in\mathbb{R\}}$ of projective IFSs that
depend continuously on a real parameter $t$. The example demonstrates how
behaviour of a projective family $F$ may be more complicated than in the
affine case. Let $\mathcal{F}(t)=(\mathbb{P}^{2};f_{1},f_{2},f_{3})$ where
\[
\begin{aligned}
f_1  & =
\begin{pmatrix}
198t+199 & 198t+198 & -198t^{2}-297t-99\\
0 & 1 & 0\\
198 & 198 & -198t-98
\end{pmatrix},\\
f_2 &  =
\begin{pmatrix}
397 & 396 & -594\\
0 & 1 & 0\\
198 & 198 & -296
\end{pmatrix}
\text{, and } f_3=
\begin{pmatrix}
595 & 594 & -1485\\
0 & 1 & 0\\
198 & 198 & -494
\end{pmatrix}.
\end{aligned}
\]
This family interpolates quadratically between three IFSs, $\mathcal{F}(0),$
$\mathcal{F}(1),$ and $\mathcal{F}(2)$, each of which has an attractor that
avoids a hyperplane. But the IFSs $\mathcal{F}(0.5)$ and $\mathcal{F}(1.5)$ do
not have an attractor$.$ This contrasts with the affine case, where similar
interpolations yield IFSs that have an attractor at all intermediate values of
the parameter. For example, if hyperbolic affine IFSs $\mathcal{F}$ and
$\mathcal{G}$ each have an attractor, then so does the average IFS, $\left(
t\mathcal{F}+(1-t)\mathcal{G}\right)  $ for all $t\in\lbrack0,1]$.

\section{\label{1implies2sec}Proof that $(1)\Rightarrow(2)$ in Theorem
\ref{main}}

\begin{lemma}
\label{(1)implies(2)} (i) If the projective IFS $\mathcal{F}$ has an attractor
$A$ then there is a nonempty open set $U$ such that $A\subset U$,
$\mathcal{F}(\overline{U})\subset U$, and $\overline{U}$ is contained in the
basin of attraction of $A$.

(ii) \textbf{[Theorem \ref{main} $\mathbf{(1)\Rightarrow(2)}$]} If the
projective IFS $\mathcal{F}$ has an attractor $A$ and there is a hyperplane
$H$ such that $H\cap A=\emptyset$, then there is a nonempty open set $U$ such
that $A\subset U,$ $\overline{U}\cap H=\emptyset,$ $\mathcal{F}(\overline
{U})\subset U$, and $\overline{U}$ is contained in the basin of attraction of
$A$.
\end{lemma}

\begin{proof}
We prove (ii) first. The proof will make use of the function $\mathcal{F}%
^{-1}(X)=\{x\in\mathbb{P}^{n}\,:\,f(x)\in X\;\text{for all}\;f\in{\mathcal{F}%
}\}$. Note that ${\mathcal{F}}^{-1}$ takes open sets to open sets,
$X\subset({\mathcal{F}}^{-1}\circ{\mathcal{F}})(X)$ and $({\mathcal{F}}%
\circ{\mathcal{F}}^{-1})(X)\subset X$ for all $X$.

Since $A$ is an attractor contained in $\mathbb{P}^{n}\setminus H$, there is
an open set $V$ containing $A$ such that $\overline{V}$ is compact,
$\overline{V}\subset\mathbb{P}^{n}\setminus H$, and $A=\lim_{k\rightarrow
\infty}{\mathcal{F}}^{k}(\overline{V})$. Hence there is an integer $m$ such
that $\mathcal{F}^{k}(\overline{V})\subset V$ for $k\geq m$.

Define $V_{k},\,k=0,1,\dots,m,$ recursively, going backwards from $V_{m}$ to
$V_{0}$, as follows. Let $V_{m}=V$ and for $k=m-1,\dots,2,1,0,$ let
$V_{k}=V\cap{\mathcal{F}}^{-1}(V_{k+1})$. If $O=V_{0}$, then $O$ has the
following properties:

\begin{enumerate}
\item $O$ is open,

\item $A\subset O$,

\item ${\mathcal{F}}^{k}(O)\subset V$ for all $k\geq0$.
\end{enumerate}

To check property (2) notice that ${\mathcal{F}}(A)=A$ implies $A\subset
({\mathcal{F}}^{-1}\circ{\mathcal{F}})(A)={\mathcal{F}}^{-1}(A)$. Then
$A\subset V=V_{m}$ implies that $A\subset V_{m}$ for all $m$, in particular
$A\subset V_{0}=O$. To check property (3) notice that $V_{k}\subset
{\mathcal{F}}^{-1}(V_{k+1})$ implies ${\mathcal{F}}(V_{k})\subset
({\mathcal{F}}\circ{\mathcal{F}}^{-1})(V_{k+1})\subset V_{k+1}$. It then
follows that ${\mathcal{F}}^{k}(O)\subset V_{k}\subset V$ for $0\leq k\leq m$.
Also ${\mathcal{F}}^{k}(O)\subset{\mathcal{F}}^{k}(V)\subset V$ for all $k>m$.

Since $A=\lim_{n\rightarrow\infty}{\mathcal{F}}^{n}(\overline{O})$, there is
an integer $K$ such that ${\mathcal{F}}^{K}(\overline{O})\subset O$. Let
$O_{k},\,k=0,1,\dots,K,$ be defined recursively, going backwards from $O_{K}$
to $O_{0}$, as follows. Let $O_{K}=O$ , and for $k=K-1,\dots,2,1,0,$ let
$O_{k}$ be an open set such that

\begin{enumerate}
\item[(4)] ${\mathcal{F}}^{k}(\overline{O})\subset O_{k}$,

\item[(5)] $\overline{O_{k}}\subset\mathbb{P}^{n}\setminus H$, and

\item[(6)] ${\mathcal{F}}(\overline{O_{k}})\subset O_{k+1}$.
\end{enumerate}

To verify that a set $O_{k}$ with these properties exists, first note that
property (4) holds for $k=K$. To verify the properties for all $k=K-1,\dots
,2,1,0$ inductively, assume that $O_{k},\,k\geq1,$ satisfies property (4).
Using property (4) we have ${\mathcal{F}}^{k-1}(\overline{O})\subset
{\mathcal{F}}^{-1}({\mathcal{F}}^{k}(\overline{O}))\subset{\mathcal{F}}%
^{-1}(O_{k})$ and using property (3) we have ${\mathcal{F}}^{k-1}(\overline
{O})\subset\overline{V}\subset\mathbb{P}^{n}\setminus H$. Now choose $O_{k-1}$
to be an open set such that ${\mathcal{F}}^{k-1}(\overline{O})\subset O_{k-1}$
and $\overline{O_{k-1}}\subset{\mathcal{F}}^{-1}(O_{k})\cap(\mathbb{P}%
^{n}\setminus H)$. The last inclusion implies ${\mathcal{F}}(\overline
{O_{k-1}})\subset O_{k}$.

We claim that
\[
U=\bigcup_{k=0}^{K-1}O_{k}%
\]
satisfies the properties in the statement of part (ii) of the lemma. (*) By
property (5) we have $\overline{U}\cap H=\emptyset$. By properties (2) and (4)
we have $A={\mathcal{F}}^{k}(A)\subset{\mathcal{F}}^{k}(O)\subset O_{k}$ for
each $k$, which implies $A\subset U$. Lastly,
\[
{\mathcal{F}}(\overline{U})=\bigcup_{k=0}^{K-1}{\mathcal{F}}(\overline{O_{k}%
})\subset\bigcup_{k=1}^{K}O_{k}=\bigcup_{k=1}^{K-1}O_{k}\cup O_{K}\subset
U\cup O\subset U\cup O_{0}\subset U,
\]
the first inclusion coming from property (6) and the second to last inclusion
coming from property (4) applied to $k=0$. This completes the proof that there
is a nonempty open set $U$ such that $A\subset U$. $\overline{U}\cap
H=\emptyset,$ and $\mathcal{F}(\overline{U})\subset U$. Now note that, by
construction, $\overline{U}$ is such that $\mathcal{F}^{K}(\overline
{U})\subset O_{K}=O$ and that $\overline{O}$ lies in $\overline{V}$ which lies
in the basin of attraction of $A,$ which implies that $\overline{U}$ is
contained in the basin of attraction of $A$. This completes the proof of (ii).

The proof of (i) is the same as the above proof of (ii), except that
$\mathbb{P}^{n}\backslash H$ is replaced by $\mathbb{P}^{n}$ throughout, and
the sentence (*) is omitted.
\end{proof}

\section{\label{proofoflemmaD}Projective transformations of convex sets}

This section describes the action of a projective transformation on a convex
set. We develop the key result, Proposition \ref{lemmaD}, that is used subsequently.

Proposition~\ref{lemmaA} states that the property of being a convex subset
(with respect to a hyperplane) of a projective space is preserved under a
projective transformation.

\begin{proposition}
\label{lemmaA} Let $f\,:\mathbb{P}^{n}\rightarrow\mathbb{P}^{n}$ be a
projective transformation. For any two hyperplanes $H,H^{\prime}$ with $S\cap
H=\emptyset$ and $f(S)\cap H^{\prime}=\emptyset$, the set $S\subset
\mathbb{P}^{n}$ is a convex set with respect to $H$ if and only if $f(S)$ is
convex with respect to $H^{\prime}.$
\end{proposition}

\begin{proof}
Assume that $S$ is convex with respect to $H$. To show that $f(S)$ is convex
with respect to $H^{\prime}$ it is sufficient to show, given any two points
$x^{\prime},y^{\prime}\in f(S)$, that $\overline{x^{\prime}y^{\prime}%
}_{H^{\prime}}\subseteq f(S)$. If $x=f^{-1}(x^{\prime})$ and $y=f^{-1}%
(y^{\prime})$, then by the convexity of $S$ and the fact that $S\cap
H=\emptyset$, we know that $\overline{xy}_{H}\subseteq S$. Hence
$f(\overline{xy}_{H})\subseteq f(S)$. Since $f(S)\cap H^{\prime}=\emptyset$,
and $f$ takes lines to lines, $\overline{x^{\prime}y^{\prime}}_{H^{\prime}%
}=f(\overline{xy}_{H})\subseteq f(S)$.

The converse follows since $f^{-1}$ is a projective transformation.
\end{proof}

Proposition~\ref{lemmaB} states that $conv_{H}(S)$ behaves well under
projective transformation.

\begin{proposition}
\label{lemmaB} Let $S\subset\mathbb{P}^{n}$ and let $H$ be a hyperplane such
that $S\cap H=\emptyset$. If $f:\mathbb{P}^{n}\rightarrow\mathbb{P}^{n}$ is a
projective transformation, then
\[
conv_{f(H)}f(S)=f(conv_{H}(S))\text{.}%
\]

\end{proposition}

\begin{proof}
Since $S\subseteq conv_{H}(S)$, we know that $f(S)\subseteq f(conv_{H}(S))$.
Moreover, by Proposition \ref{lemmaA}, we know that $f(conv_{H}(S))$ is convex
with respect to $f(H)$. To show that $conv_{f(H)}f(S)=f(conv_{H}(S))$ it is
sufficient to show that $f(conv_{H}(S))$ is the smallest convex subset
containing $f(S)$, i.e., there is no set $C$ such that $C$ is convex with
respect to $f(H)$ and $f(S)\subseteq C\subsetneq f(conv_{H}(S))$. However, if
such a set exists, then by applying the inverse $f^{-1}$ to the above
inclusion, we have $S\subseteq f^{-1}(C)\subsetneq conv_{H}(S)$. Since
$f^{-1}(C)$ is convex by Proposition \ref{lemmaA}, we arrive at a
contradiction to the fact that $conv_{H}(S)$ is the smallest convex set
containing $S$.
\end{proof}

In general, $conv_{H}(S)$\textit{ }depends on the avoided\textit{
}hyperplane\textit{ }$H$\textit{.} But, as Proposition \ref{lemmaC} shows, it
is independent of the avoided hyperplane when $S$ is connected.

\begin{proposition}
\label{lemmaC} If $S\subset\mathbb{P}^{n}$ is a connected set such that $S\cap
H=S\cap H^{\prime}=\emptyset$ for hyperplanes $H,H^{\prime}$ of $\mathbb{P}%
^{n}$, then
\[
conv_{H}(S)=conv_{H^{\prime}}(S).
\]

\end{proposition}

\begin{proof}
The fact that $S$ is connected and $S\cap H^{\prime}=\emptyset$, implies that
$conv_{H}(S)\cap H^{\prime}=\emptyset.$ Therefore $conv_{H}(S)$ is the
ordinary convex hull of $S$ in $({\mathbb{P}}^{n}\setminus H)\backslash
H^{\prime},$ which is an affine $n$-dimensional space with a hyperplane
deleted. Likewise $conv_{H^{\prime}}(S)$ is the ordinary convex hull of $S$ in
$(\mathbb{P}^{n}\setminus H^{\prime})\backslash H=(\mathbb{P}^{n}\setminus
H)\backslash H^{\prime}$. Therefore $conv_{H}(S)=conv_{H^{\prime}}(S)$.
\end{proof}

The key result, that will be needed, for example in Section \ref{2implies3sec}%
, is the following.

\begin{proposition}
\label{lemmaD} Let $S\subset\mathbb{P}^{n}$ be a connected set and let $H$ be
a hyperplane. If $S\cap H=\emptyset$ and $f:\mathbb{P}^{n}\rightarrow
\mathbb{P}^{n}$ is a projective transformation such that $f(S)\cap
H=\emptyset,$ then
\[
conv_{H}f(S)=f(conv_{H}(S))\text{.}%
\]

\end{proposition}

\begin{proof}
This follows at once from Propositions \ref{lemmaB} and \ref{lemmaC}.
\end{proof}

\section{\label{2implies3sec}Proof that (2)$\Rightarrow$(3) in Theorem
\ref{main}}

The implication (2)$\Rightarrow$(3) in Theorem \ref{main} is proved in two
steps. We show that (2)$\Rightarrow$(2.5)$\Rightarrow$(3) where (2.5) is the
following statement.\smallskip

(2.5) \textit{There is a hyperplane }$H$\textit{ and nonempty finite
collection of nonempty disjoint connected open sets }$\left\{  O_{i}\right\}
$\textit{ such that }$\mathcal{F}(\cup_{i}\overline{O}_{i})\subset\cup
_{i}O_{i}$\textit{ and }$\cup_{i}\overline{O}_{i}\cap H=\emptyset.$\smallskip

\begin{lemma}
\label{(2)implies(3)A}\textbf{[(2)$\Rightarrow$(2.5)]} If there is a nonempty
open set $U$ and a hyperplane $H$ with $\overline{U}\cap H=\emptyset$ such
that $\mathcal{F}(\overline{U})\subset U$, then there is a nonempty finite
collection of nonempty disjoint connected open sets $\left\{  O_{i}\right\}  $
such that $\mathcal{F}(\cup_{i}\overline{O}_{i})\subset\cup_{i}O_{i}$ and
$\cup_{i}\overline{O}_{i}\cap H=\emptyset$.
\end{lemma}

\begin{proof}
Let $U=\cup_{\alpha}U_{\alpha}$, where the $U_{\alpha}$ are the connected
components of $U$. Let $\tilde{A}=\cap_{k}\mathcal{F}^{k}(\overline{U})$ and
let $\left\{  O_{i}\right\}  $ be the set of $U_{\alpha}$ that have nonempty
intersection with $\tilde{A}$. This set is finite because the sets in
$\left\{  O_{i}\right\}  $ are pairwise disjoint and $\tilde{A}$ is compact.
Since $\mathcal{F}(\tilde{A})\subset\tilde{A}$ and $\mathcal{F}(\overline
{U})\subset U$, we find that $\mathcal{F}(\cup\overline{O}_{i})\subset\cup
_{i}O_{i}$. Since $\cup_{i}\overline{O}_{i}\subset\overline{U}$ and
$\overline{U}\cap H=\emptyset$, we have $\cup_{i}\overline{O}_{i}\cap
H=\emptyset$.
\end{proof}

\begin{lemma}
\label{(2)implies(3)B}[\textbf{(2.5)}$\Rightarrow$\textbf{(3)}]: If there is a
nonempty finite collection of nonempty disjoint connected open sets $\left\{
O_{i}\right\}  $ and a hyperplane $H$ such that $\mathcal{F}(\cup_{i}%
\overline{O}_{i})\subset\cup_{i}O_{i}$ and $\cup_{i}\overline{O}_{i}\cap
H=\emptyset,$ then there is a nonempty finite collection of disjoint convex
bodies $\left\{  C_{i}\right\}  $ such that $\mathcal{F}(\cup_{i}C_{i})\subset
int(\cup_{i}C_{i})$.
\end{lemma}

\begin{proof}
Assume that there is a nonempty finite collection of nonempty disjoint
connected open sets $\left\{  O_{i}\right\}  $ such that $\mathcal{F}(\cup
_{i}\overline{O}_{i})\subset\cup_{i}O_{i}$ and $\cup_{i}\overline{O}_{i}$
avoids a hyperplane. Let $O=\cup_{i}O_{i}$. Since $\mathcal{F}(\overline
{O})\subset O$, it must be the case that, for each $f\in\mathcal{F}$ and each
$i$, there is an index that we denote by $f(i)$, such that $f(\overline{O_{i}%
})\subset O_{f(i)}$. Since $\overline{O_{i}}$ is connected and both
$\overline{O_{i}}$ and $f(\overline{O_{i}})$ avoid the hyperplane $H$ it
follows from Proposition \ref{lemmaD} that%
\[
f(conv_{H}(\overline{O_{i}}))=conv_{H}(f(\overline{O_{i}}))\subset
conv_{H}(O_{f(i)})\subset int(conv_{H}(\overline{O_{f(i)}})).
\]
For each $i$, let $C_{i}=conv_{H}(\overline{O}_{i})$, so that each $C_{i}$ is
a convex body. Then we have%
\[
f(C_{i})\subset int(C_{f(i)}).
\]

However, it may occur, for some $i\neq j$, that $C_{i}\cap C_{j}\neq\emptyset
$. In this case $C_{i}\cup C_{j}$ is a connected set that avoids the
hyperplane $H$, and is such that $f(C_{i}\cup C_{j})$ also avoids $H.$ It
follows again by Proposition \ref{lemmaD} that
\[
conv_{H}(f(C_{i}\cup C_{j}))=f(conv_{H}(C_{i}\cup C_{j}))\subset
int(conv(C_{f(i)}\cup C_{f(j)}).
\]

Define $C_{i}$ and $C_{j}$ to be related if $C_{i}\cap C_{j}\neq\emptyset$,
and let $\sim$ denote the transitive closure of this relation. (That is, if
$C_{i}$ is related to $C_{j}$ and $C_{j}$ is related to $C_{k}$, then $C_{i}$
is related to $C_{k}$.) From the set $\{C_{i}\}$ define a new set $U^{\prime}$
whose elements are
\[
U^{\prime}=\left\{  conv\left(  \bigcup_{C\in Z}C\right)
\,:\,Z\;\mbox{is an equivalence class with respect to $\sim$}\right\}  .
\]
By abuse of language, let $\{C_{i}\}$ be the set of convex sets in $U^{\prime
}$. It may again occur, for some $i\neq j$, that $C_{i}\cap C_{j}\neq
\emptyset$. In this case we repeat the equivalence process. In a finite number
of such steps we arrive at a finite set of disjoint convex bodies $\{C_{i}\}$
such that\textbf{ }$\mathcal{F}(\cup C_{i})\subset int(\cup C_{i})$.
\end{proof}

Lemma \ref{(2)implies(3)A} and Lemma \ref{(2)implies(3)B} taken together imply
that $(2)\Rightarrow(3)$ in Theorem \ref{main}.

\section{\label{p1of3implies4sec}Part 1 of the proof that $(3)\Rightarrow(4)$
in Theorem \ref{main}}

The standing assumption in this section is that statement $(3)$ of Theorem
\ref{main} is true. We begin to develop a metric with respect to which
$\mathcal{F}$ is contractive. The final metric is defined in the next section.

Let $\mathcal{U}:=\{C_{1},C_{2},...,C_{q}\}$ be the set of nonempty convex
connected components in statement (3) of Theorem \ref{main}. Define a directed
graph (digraph) $G$ as follows. The nodes of $G$ are the elements of
$\mathcal{U}$. For each $f\in\mathcal{F}$, there is an edge \textit{colored}
$f$ directed from node $U$ to node $V$ if $f(U)\subset int(V)$. Note that, for
each node $U$ in $G$, there is exactly one edge of each color emanating from
$U$. Note also that $G$ may have \textit{multiple edges} from one node to
another and may have loops. (A \textit{loop} is an edge from a node to itself.)

A \textit{directed path} in a digraph is a sequence of nodes $U_{0}%
,U_{1},\dots,U_{k}$ such that there is an edge directed from $U_{i-1}$ to
$U_{i}$ for $i=1,2\dots,k$. Note that a directed path is allowed to have
repeated nodes and edges. Let $p=U_{0},U_{1},\dots,U_{k}$ be a directed path.
If $f_{1},f_{2},\dots,f_{k}$ are the colors of the successive edges, then we
will say that $p$ has \textit{type} $f_{1}\,f_{2}\cdots\,f_{k}$.

\begin{lemma}
\label{lem7} The graph $G$ cannot have two directed cycles of the same type
starting at different nodes.
\end{lemma}

\begin{proof}
By way of contradiction assume that $U\neq U^{\prime}$ are the starting nodes
of two paths $p$ and $p^{\prime}$ of the same type $f_{1}\,f_{2}\cdots\,f_{k}%
$. Recall that the colors are functions of the IFS $\mathcal{F}$. If
$g=f_{k}\circ f_{k-1}\circ\cdots\circ f_{1}\circ f_{0}$, then the composition
$g$ takes the convex set $U$ into $int(U)$ and the convex set $U^{\prime}$
into $int(U^{\prime})$. By the Krein-Rutman theorem \cite{KreinRutman} this is
impossible. More specifically, the Krein-Rutman theorem tells us that if $K$
is a closed convex cone in ${\mathbb{R}} ^{n+1}$ and $L \, :\, {\mathbb{R}}
^{n+1} \rightarrow{\mathbb{R}} ^{n+1}$ is a linear transformation such that
$L(K) \subset int(K)$, then the spectral radius $r(L) > 0$ is a simple
eigenvalue of $L$ with an eigenvector $v\in int(K)$.
\end{proof}

Each function $f\in\mathcal{F}$ acts on the set of nodes of $G$ in this way:
$f(U)=V$ where $(U,V)$ is the unique edge of color $f$ starting at $U$.

\begin{lemma}
\label{metric} There exists a metric $d_{G}$ on the set of nodes of $G$ such that

\begin{enumerate}
\item $d_{G}(U,V) \geq2$ for all $U\neq V$ and

\item each $f\in\mathcal{F}$ is a contraction with respect to $d_{G}$.
\end{enumerate}
\end{lemma}

\begin{proof}
Starting from the graph $G$, construct a directed graph $G_{2}$ whose set of
nodes consists of all unorder pairs $\{U,V\}$ of distinct nodes of $G$. In
$G_{2}$ there is an edge from $\{U,V\}$ to $\{f(U),f(V)\}$ for all nodes
$\{U,V\}$ in $G_{2}$ and for each $f\in\mathcal{F}$. Since $G$ has no two
directed cycles of the same type starting at different nodes, we know by Lemma
\ref{lem7} that $G_{2}$ has no directed cycle. Because of this, a partial
order $\prec$ can be defined on the node set of $G_{2}$ by declaring that
$\{U^{\prime},V^{\prime}\}\prec\{U,V\}$ if there is an edge from $\{U,V\}$ to
$\{U^{\prime},V^{\prime}\}$ and then taking the transitive closure. Every
finite partially ordered set has a linear extension (see \cite{brualdi} for
example), i.e. there is an ordering $<$ of the nodes of $G_{2}$:
\[
\{U_{1},V_{1}\}<\{U_{2},V_{2}\}<\cdots<\{U_{m},V_{m}\}
\]
such that if $\{U,V\}\prec\{U^{\prime},V^{\prime}\}$ then $\{U,V\}<\{U^{\prime
},V^{\prime}\}$. Using $N(G)$ to denote the set of nodes of $G$, define a map
$d_{G}:\,N(G)\times N(G)\rightarrow\lbrack0,\infty)$ in any way satisfying

\begin{enumerate}
\item $d_{G} (U,U)=0$ for all $U\in N(G),$

\item $d_{G} (U,V)=d_{G} (V,U)$ for all $U,V\in N(G)$, and

\item $2\leq d_{G}(U_{1},V_{1})<d_{G}(U_{2},V_{2})<\cdots<d_{G}(U_{m},V_{m})
\leq4$.
\end{enumerate}

Properties (1), (2) and (3) guarantee that $d_{G}$ is a metric on $N(G)$. The
fact $2\leq d_{G}(U_{i},V_{i})\leq4$ for all $i$ guarantees the triangle
inequality. If
\[
s=\min_{1\leq i<m}\,\frac{d_{G}(U_{i},V_{i})}{d_{G}(U_{i+1},V_{i+1})},
\]
then $0<s<1$ and, for any $f\in\mathcal{F}$, we have
\[
d_{G}(f(U),f(V))\leq s\,d_{G}(U,V)
\]
because $\{f(U),f(V)\}\prec\{U,V\}$ by the definition of the partial order and
\linebreak$\{f(U),f(V)\}<\{U,V\}$ by the definition of linear extension. Hence
$f$ is a contraction with respect to $d_{G}$ for any $f\in\mathcal{F}$.
\end{proof}

\section{\label{p2of3implies4sec}Part 2 of the proof that (3)$\Rightarrow$(4)
in Theorem \ref{main}}

In this section we construct a metric $d_{i}$ on each component $C_{i}$ of the
collection $\left\{  C_{i}\right\}  =\left\{  C_{i}:i=1,2,...,q\right\}  $ in
statement (3) of Theorem \ref{main}$.$ We will then combine the metrics
$d_{i}$ with the graph metric $d_{G}$ in Section \ref{p1of3implies4sec} to
build a metric on $\cup_{i}C_{i}$ such that statement (4) in Theorem
\ref{main} is true. Proofs that a projective transformation is contractive
with respect to the Hilbert metric go back to G. Birkhoff \cite{birkhoff};
also see P. J. Bushell \cite{bushell}. The next lemma is used to compute the
contraction factor for projective maps under the Hilbert Metric.

\begin{lemma}
\label{HilbertLemma2} If $r \ge\alpha\ge0, t \ge\alpha,$ and $h,h^{\prime
},s,s^{\prime}\in(0,1),$ where $s^{\prime}= 1-s, h^{\prime}= 1 - h,$ and $s
\le h,$ then $\log(\frac{(r +h)(t +s^{\prime})}{(r +s)(t +h^{\prime})})
\le\log(\frac{(\alpha+h)(\alpha+s^{\prime})}{(\alpha+s)(\alpha+h^{\prime})})
\le\frac{1}{\alpha+1} \log(\frac{hs^{\prime}}{sh^{\prime}}).$
\end{lemma}

\begin{proof}
Since we are assuming that $s \le h, s(1-h) >0,$ and $\alpha\ge0,$ it is an
easy exercise to show that $\frac{(\alpha+h)(\alpha+s^{\prime})}%
{(\alpha+s)(\alpha+h^{\prime})} \ge1.$ A bit of algebra can be used to show
that $N:= \frac{(\alpha+h)(\alpha+s^{\prime})}{(\alpha+s)(\alpha+h^{\prime})}
= \frac{ (1 - \frac{h^{\prime}}{\alpha+1}) (1 - \frac{s }{\alpha+1} ) } { (1 -
\frac{s^{\prime}}{\alpha+1}) (1 - \frac{h }{\alpha+1} ) }.$ If we let $\alpha=
0$ in the above expression, we observe that $D:=\frac{h s^{\prime}}{s
h^{\prime}} = \frac{ (1 - h^{\prime})(1-s) }{(1-s^{\prime})(1-h)}.$

Since $\ln(1-x) = \log_{e}(1-x) = -\sum_{j=1}^{\infty} \frac{x^{j}}{j},$
whenever $|x| < 1,$ for a logarithm of any base we see that
\begin{align*}
\frac{\log(N)}{\log(D)} = \  &  \frac{ \log(1 - \frac{h^{\prime}}{\alpha+1}) +
\log(1 - \frac{s }{\alpha+1} ) - \log(1 - \frac{s^{\prime}}{\alpha+1}) -
\log(1 - \frac{h }{\alpha+1} ) } { \log(1 - h^{\prime}) + \log(1-s) -
\log(1-s^{\prime}) - \log(1-h) }\\
= \  &  \frac{ -\sum_{j=1}^{\infty} \big[ \frac{h^{\prime j}}{j(\alpha+1)^{j}
} + \frac{s^{j}}{j(\alpha+1)^{j} } - \frac{s^{\prime j}}{j(\alpha+1)^{j} } -
\frac{h^{j}}{j(\alpha+1)^{j} } \big] } { -\sum_{j=1}^{\infty} \big[ \frac
{h^{\prime j}}{j} + \frac{s^{j}}{j} - \frac{s^{\prime j}}{j} - \frac{h^{j}}{j}
\big] }\\
= \  &  \frac{1}{\alpha+1} \frac{ \sum_{j=1}^{\infty} \frac{1}{(\alpha
+1)^{j-1}} \big[ \frac{s^{\prime j}}{j } + \frac{h^{j}}{j } - \frac{h^{\prime
j}}{j } - \frac{s^{j}}{j } \big] } { \sum_{j=1}^{\infty} \big[ \frac{s^{\prime
j}}{j}+ \frac{h^{j}}{j} - \frac{h^{\prime j}}{j} - \frac{s^{j}}{j} \big] }\\
\le\  &  \frac{1}{\alpha+1}.
\end{align*}
Note that the above inequality holds because the assumption $s \le h$ implies
$s^{\prime}= 1 - s \ge1-h = h^{\prime}$ and $(1-s)^{j} + h^{j} \ge(1-h)^{j} +
s^{j},$ for all positive integers $j.$ Thus, the series in the numerator and
denominator can be compared term by term. Finally, it is a straightforward
argument to show the numerator $N(\alpha)$ has the property that if $r
\ge\alpha\ge0$ and $t \ge\alpha\ge0,$ then $\frac{(r +h)(t +s^{\prime})}{(r
+s)(t +h^{\prime})} \le\frac{(\alpha+h)(\alpha+s^{\prime})}{(\alpha
+s)(\alpha+h^{\prime})}.$ Thus, $\log(\frac{(r +h)(t +s^{\prime})}{(r +s)(t
+h^{\prime})}) \le\log(\frac{(\alpha+h)(\alpha+s^{\prime})}{(\alpha
+s)(\alpha+h^{\prime})}).$
\end{proof}

\begin{proposition}
\label{Hilbert} Let $\mathcal{F}$ be a projective IFS and let there be a
nonempty finite collection of disjoint convex bodies $\left\{  C_{i}%
:i=1,2,...,q\right\}  $ such that $\mathcal{F}(\cup_{i}C_{i})\subset
int(\cup_{i}C_{i})$ as in statement (3) of Theorem \ref{main}. For
$i\in\{1,2,...q\}$ and $f\in\mathcal{F}$, let $f(i)\in\left\{
1,2,...,q\right\}  $ be defined by $f(C_{i})\subset C_{f(i)}$. \ Then there is
a metric $d_{i}$ on $C_{i}$, giving the same topology on $C_{i}$ as
$d_{\mathbb{P}}$, such that

1. $(C_{i},d_{i})$ is a complete metric space, for all $i=1,2,...,q;$

2. there is a real $0\leq\alpha<1$ such that
\[
d_{f(i)}(f(x),f(y))\leq\alpha d_{i}(x,y)
\]
for all $x,y\in C_{i}$, for all $i=1,2,...q,$ for all $f\in\mathcal{F}$; and

3. $d_{i}(x,y)\leq1$ for all $x,y\in C_{i}$ and all $i=1,2,...q$.
\end{proposition}

\begin{proof}
For each $C_{i}$ there exists a hyperplane $H_{i}$ such that $H_{i}\cap
C_{i}=\emptyset$. Let $\widehat{C}_{i}=\{x\in\mathbb{P}^{n}:d_{\mathbb{P}%
}(x,y)\leq\varepsilon,y\in C_{i}\}$ where $\varepsilon$ is chosen so small
that (i) $H_{i}\cap\widehat{C}_{i}=\emptyset$; and (ii) $f(\widehat{C}%
_{i})\subset int(\widehat{C}_{f(i)})$ $\forall f\in\mathcal{F},\forall
i\in\left\{  1,2,...q\right\}  $.

Given arbitrary $x,y\in int\left(  \widehat{C}_{i}\right)  $, let $a,b$ be the
points where the line $\overline{xy}$ intersects $\partial\widehat{C}_{i}$ and
let $a_{f},b_{f}$ be the points where the line $\overline{f(x)f(y)}$
intersects $\partial\widehat{C}_{f(i)}$. Let $\hat{d}_{i}$ denote the Hilbert
metric on the interior of $\widehat{C}_{i}$ for each $i$, and define
\[
\beta_{f,i}=\min\{|xy|:x\in\partial\widehat{C}_{f(i)},y\in f(\widehat{C}%
_{i})\}>0,\text{ for }f\in\mathcal{F}\text{,} i\in\left\{  1,2,...q\right\}
\text{.}
\]
We claim that
\begin{align}
\hat{d}_{f(i)}(f(x),f(y))  &  =\ln\left(  \frac{|a_{f}\,f(y)|\,|f(x)\,b_{f}%
|}{|a_{f}\,f(x)|\,|f(y)\,b_{f}|}\right) \label{ineq1}\\
&  \leq\frac{1}{\beta_{f,i}+1}\,\ln\left(  \frac{|f(a)\,f(y)|\,|f(x)\,f(b)|}%
{|f(a)\,f(x)|\,|f(y)\,f(b)|}\right) \nonumber\\
&  =\frac{1}{\beta_{f,i}+1}\,\ln\left(  \frac{|a\,y|\,|x\,b|}{|a\,x|\,|y\,b|}%
\right)  =\frac{1}{\beta_{f,i}+1}\,\hat{d}_{i}(x,y)\text{,}\nonumber
\end{align}
for all $x,y\in int\left(  \widehat{C}_{i}\right)  $, for all $f\in
\mathcal{F}$, and all $i=1,2,...$. Here $|\cdot|$ denotes Euclidean distance
as discussed in Section \ref{hilbert}. The second to last equality is the
invariance of the cross ratio under a projective transformation. Concerning
the inequality, let, without loss of generality, $|f(a)\,f(b)|=1$ and let
$h:=|f(a)\,f(y)|$ and $s:=|f(a)\,f(x)|$. Moreover let $r:=|a_{f}\,f(x)|$ and
$t:=|f(y)\,b_{f}|$. Finally let $s^{\prime}=1-s$ and $h^{\prime}=1-h$. Note
that $s\leq h<1$. The inequality is now the inequality of
Lemma~\ref{HilbertLemma2}.

Now let $\alpha=\max\{\frac{1}{1+\beta_{f,i}}:f\in\mathcal{F}$, $\forall$
$i=1,2,...q\}<1$. It follows that
\[
\hat{d}_{f(i)}(f(x),f(y))\leq\alpha\hat{d}_{i}(x,y)
\]
for all $x,y\in\widehat{C}_{i}$, for all $i=1,2,...q,$ for all $f\in
\mathcal{F}.$ Since $C_{i}\subset int\left(  \widehat{C}_{i}\right)  $ it
follows that statement (2) in Proposition \ref{Hilbert} is true.

Statement (1) follows at once from the fact the topology generated by the
Hilbert metric $\hat{d}_{i}$ on $C_{i}$ as defined above is bi-Lipschitz
equivalent to $d_{\mathbb{P}}$; see Remark \ref{remark4}.

Since $\hat{d}_{i}:C_{i}\times C_{i}\rightarrow\mathbb{R}$ is continuous and
$C_{i}\times C_{i}$ is compact, it follows that there is a constant $J_{i}$
such that $\hat{d}_{i}(x,y)\leq J_{i}$ for all $x,y\in C_{i}$. Let $J=\max
_{i}J_{i}$, and define a new metric $d_{i}$ by $d_{i}(x,y)=\hat{d}_{i}(x,y)/J$
for all $x,y\in C_{i}$. We have that $d_{i}$ satisfies (\textit{1), (2)} and
(3) in the statement of Proposition \ref{Hilbert}.
\end{proof}

\begin{lemma}
\label{(3)implies(4)lemma} \textbf{[Theorem \ref{main} (3)}$\Rightarrow
$\textbf{(4)]:} If there \textit{is a nonempty finite collection of disjoint
convex bodies }$\left\{  C_{i}\right\}  $\textit{ such that } $\mathcal{F}%
(\cup_{i}C_{i})\subset int(\cup_{i}C_{i})$\textit{, as in statement (3) of
Theorem \ref{main}}, then there is a nonempty open set $U\subset\mathbb{P}%
^{n}$ and a metric $d\,:\overline{U}\rightarrow\lbrack0,\infty)$, generating
the same topology as $d_{\mathbb{P}}$ on $\overline{U}$, such that
$\mathcal{F}$ is contractive on $\overline{U}$.
\end{lemma}

\begin{proof}
Let $U=\cup_{i}int(C_{i}).$ Define $d:\overline{U}\times\overline{U}$ by%
\[
d(x,y)=\left\{
\begin{array}
[c]{ll}%
d_{i}(x,y) & \text{if }\left(  x,y\right)  \in C_{i}\times C_{i}\text{ for
some }i,\\
d_{G}(C_{i},C_{j}) & \text{if }\left(  x,y\right)  \in C_{i}\times C_{j}\text{
for some }i\neq j,
\end{array}
\right.
\]
where the metrics $d_{i}$ and $d_{G}$ are defined in Lemma \ref{metric} and
Proposition \ref{Hilbert}.

First we show that $d$ is a metric on $\overline{U}$. We only need to check
the triangle inequality. If $x,y$ and $z$ lie in the same connected component
of $C_{i}$, the triangle inequality follows from Proposition \ref{Hilbert}. If
$x,y$ and $z$ lie in three distinct components, the triangle inequality
follows from Lemma~\ref{metric}. If $x,y\in C_{i}$ and $z\in C_{j}$ for some
$i\neq j$, then
\begin{align*}
d(x,y)+d(y,z)  &  =d_{i}(x,y)+d_{G}(C_{i},C_{j})\geq d_{G}(C_{i},C_{j})
=d(x,z),\\
d(x,z)+d(z,y)  &  =d_{G}(C_{i},C_{j})+d_{G}(C_{j},C_{i})\geq2\geq
d_{i}(x,y)=d(x,y).
\end{align*}

Second we show that $\mathcal{F}$ is contractive with respect to $d$. By
Proposition \ref{Hilbert} there is $0\leq\alpha<1$ such that, if $x$ and $y$
lie in the same connected component of $U$ and $f\in\mathcal{F}$, then
\[
d(f(x),f(y))\leq\alpha\,d(x,y).
\]
If $x$ and $y$ lie in different connected components of $U$, then there are
two cases. If $f(x)$ and $f(y)$ lie in different connected components, then by
Lemma \ref{metric},
\[
d(f(x),f(y))=d_{G}(f(x),f(y))\leq\alpha_{G}\,d_{G}(x,y)=d(x,y),
\]
where $\alpha_{G}$ is the constant guaranteed by Lemma \ref{metric}. If $f(x)$
and $f(y)$ lie in the same connected component $U_{i}$, then
\[
d(f(x),f(y))=d_{i}(f(x),f(y))\leq1\leq\frac{1}{2}\,d_{G}(x,y)=\frac{1}%
{2}d(x,y).
\]
Third, and last, the metric $d$ generates the same topology on $\overline{U}$
as the metric $d_{\mathbb{P}}$, because, for any convex body $K$, the Hilbert
metric $d_{K}$ and the metric $d_{\mathbb{P}}$ are bi-Lipshitz equivalent on
any compact subset of the interior of $K$; see Remark \ref{remark4} in Section
\ref{remarksec}.
\end{proof}

\section{\label{(4)implies(1)sec}Proof that (4)$\Rightarrow$(1) in Theorem 1
and the Proof of the Uniqueness of Attractors}

This section contains a proof that statement (4) implies statement (1) in
Theorem \ref{main} and a proof of Theorem {\ref{uniquethm} on the uniqueness
of the attractor. }

A point $p_{f}\in\mathbb{P}^{n}$ is said to be an \textit{attractive fixed
point} of the projective transformation $f$ if $f(p_{f})=p_{f}$, and $f$ is a
contraction with respect to the round metric on some open ball centered at
$p_{f}$. If $f$ has an attractive fixed point, then the real Jordan canonical
form \cite{jordan} can be used to show that any matrix $L_{f}\,:\,{\mathbb{R}%
}^{n+1}\rightarrow{\mathbb{R}}^{n+1}$ representing $f$ has a dominant
eigenvalue. In the case that $f$ has an attractive fixed point, let $E_{f}$
denote the $n$-dimensional $L_{f}$-invariant subspace of ${\mathbb{R}}^{n+1}$
that is the span of the eigenspaces corresponding to all the other
eigenvalues. Let $H_{f}:=\phi(E_{f})$ be the corresponding hyperplane in
$\mathbb{P}^{n}$. Note that $H_{f}$ is invariant under $f$ and $p_{f}\notin
H_{f}$. Moreover, the basin of attraction of $p_{f}$ for $f$ is ${\mathbb{P}%
}^{n}\setminus H_{f}$.

\begin{lemma}
\label{4implies1} \textbf{[Theorem~\ref{main} (4)} $\mathbf{\Rightarrow}$
\textbf{(1)]:} If there is a nonempty open set $U\subset\mathbb{P}^{n}$ such
that $\mathcal{F}$ is contractive on $\overline{U}$, then $\mathcal{F}$ has an
attractor $A$ that avoids a hyperplane.
\end{lemma}

\begin{proof}
We are assuming statement (4) in Theorem \ref{main} that the IFS $\mathcal{F}$
is contractive on $\overline{U}$ with respect to some metric $d$.. Since
$\overline{U}$ is compact and $({\mathbb{P}}^{n},d_{{\mathbb{P}}})$ is a
complete metric space, $(\overline{U},d)$ is a complete metric space. It is
well known in this case \cite{hutchinson} that $\mathcal{F}$ has an attractor
$A\subset U$. It only remains to show that there is a hyperplane $H$ such that
$A\subset\mathbb{P}^{n}\setminus H$.

Let $f$ be any function in $\mathcal{F}$. Since $f$ is a contraction on
$\overline{U}$, we know by the Banach contraction mapping theorem that $f$ has
an attractive fixed point $x_{f}$. We claim that $x_{f}\in A$. If
$x\in{\mathbb{P}}^{n}\setminus H_{f}$ lies in the basin of attraction of $A$,
then $x_{f}=\lim_{k\rightarrow\infty}f^{k}(x)\in A$. It now suffices to show
that $A\cap H_{f}=\emptyset$. By way of contradiction, assume that $x\in A\cap
H_{f}$. Since $\mathcal{F}$ is contractive on $\overline{U}$, it is
contractive on $A$. Since $x_{f}\in A$, we have $d(f^{k}(x),x_{f}%
)=d(f^{k}(x),f^{k}(x_{f}))\rightarrow0$ as $k\rightarrow\infty$, which is
impossible since $f^{k}(x)\in H_{f}$ and $x_{f}\notin H_{f}$.
\end{proof}

So now we have that Statements (1), (2), (3) and (4) in Theorem \ref{main} are
equivalent. The proof of Lemma~\ref{4implies1} also shows the following.

\begin{corollary}
\label{cor} If $\mathcal{F}$ is a contractive IFS, then each $f\in\mathcal{F}$
has an attractive fixed point $x_{f}$ and an invariant hyperplane $H_{f}$.
\end{corollary}

\begin{proposition}
\label{good} Let $\mathcal{F}$ be a projective IFS containing at least one map
that has an attractive fixed point. If $\mathcal{F}$ has an attractor $A$,
then $A$ is the unique attractor in $\mathbb{P}^{n}$.
\end{proposition}

\begin{proof}
Assume that there are two distinct attractors $A$, $A^{\prime}$, and let $U,$
$U^{\prime}$ be their respective basins of attraction. If $U\cap U^{\prime
}\neq\emptyset$, then $A=A^{\prime}$, because if there is $x\in U\cap
U^{\prime}$ then $A^{\prime}=\lim_{k\rightarrow\infty}\mathcal{F}^{k}(x)=A,$
where the limit is with respect to the Hausdorff metric. Therefore $U\cap
U^{\prime}=\emptyset$ and $A\cap A^{\prime}=\emptyset$.

If $f\in\mathcal{F}$ has an attractive fixed point $p_{f}$ and $p\in
U\setminus H_{f},$ and $p^{\prime}\in U^{\prime}\setminus H_{f}$, then both
\begin{align*}
p_{f}  &  =\lim_{k\rightarrow\infty}f^{k}(p)\subseteq\lim\mathcal{F}%
^{k}(p)=A,\text{ and}\\
p_{f}  &  =\lim_{k\rightarrow\infty}f^{k}(p^{\prime})\subseteq\lim
\mathcal{F}^{k}(p^{\prime})=A^{\prime}.
\end{align*}
But this is impossible since $A\cap A^{\prime}=\emptyset$. So Proposition
\ref{good} is proved.
\end{proof}

We can now prove Theorem~\ref{uniquethm} - that a projective IFS has at most
one attractor.

\begin{proof}
[\textbf{Proof of Theorem 2}]Assume, by way of contradiction, that $A$ and
$A^{\prime}$ are distinct attractors of ${\mathcal{F}}$ in ${\mathbb{P}}^{n}$.
As in the proof of Proposition \ref{good}, it must be the case that $A\cap
A^{\prime}=\emptyset$ and hence that their respective basins of attraction are disjoint.

By Lemma \ref{(1)implies(2)} there exist open sets $U$ and $U^{\prime}$ such
that $A\subset U,\,A^{\prime}\subset U^{\prime},$ and ${\mathcal{F}}%
(\overline{U})\subset U$ and ${\mathcal{F}}(\overline{U^{\prime}})\subset
U^{\prime}$. Since $\overline{U}$ belongs to the basin of attraction of $A$
and $\overline{U^{\prime}}$ belongs to the basin of attraction of $A^{\prime
},$ we have $U\cap U^{\prime}=\emptyset.$ If $f\in{\mathcal{F}}$ and $x\in U$,
then in the Hausdorff topology
\[
A(x):=\lim_{k\rightarrow\infty}\overline{\cup_{m\geq k}f^{m}(x)}\subset A
\]
and $A(x)$ is nonempty. Similarly, if $x^{\prime}\in U^{\prime}$, then
\[
A(x^{\prime}):=\lim_{k\rightarrow\infty}\overline{\cup_{m\geq k}%
f^{m}(x^{\prime})}\subset A^{\prime}%
\]
and $A(x^{\prime})$ is nonempty.

Let $L_{f}$ be a matrix for $f\in{\mathcal{F}}$ in real Jordan canonical form
and such that the largest modulus of an eigenvalue is $1$. Let $W$ denote the
$L_{f}$-invariant subspace of ${\mathbb{R}}^{n+1}$ corresponding to the
eigenvalues of modulus $1$, and let $L$ denote the restriction of $L_{f}$ to
$W$. If $E$ is the subspace of ${\mathbb{P}}^{n}$ corresponding to the
subspace $W$ of ${\mathbb{R}}^{n+1}$, then, by use of the Jordan canonical
form, $A(x)\subset E$ and $A(x^{\prime})\subset E$. Together with the
inclusions above, this implies that $A\cap E\neq\emptyset$ and $A^{\prime}\cap
E\neq\emptyset$. Hence $U_{E}:=U\cap E\neq\emptyset$ and $U_{E}^{\prime
}:=U^{\prime}\cap E\neq\emptyset$ and if $f|_{E}$ denotes the restriction of
$f$ to $E$, then
\begin{equation}
f|_{E}(\overline{U_{E}})=f(\overline{U}\cap E)=f(\overline{U})\cap E\subset
U\cap E=U_{E}, \label{restriction-equation}%
\end{equation}
and similarly $f|_{E}(\overline{U_{E}^{\prime}})\subset U_{E}^{\prime}.$

Each Jordan block of $L$ can have one of the following forms%

\[
(a) \quad%
\begin{pmatrix}
\lambda & 0 & \cdots & 0\\
0 & \lambda & \cdots & 0\\
&  & \ddots & \\
0 & 0 & \cdots & \lambda
\end{pmatrix}
\qquad\qquad(b) \quad%
\begin{pmatrix}
R & {\mathbf{0}} & \cdots & {\mathbf{0}}\\
{\mathbf{0}} & R & \cdots & {\mathbf{0}}\\
&  & \ddots & \\
{\mathbf{0}} & {\mathbf{0}} & \cdots & R
\end{pmatrix}
\]
\vskip 3mm%

\[
(c) \quad%
\begin{pmatrix}
\lambda & 1 & 0 & \cdots & 0\\
0 & \lambda & 1 & \cdots & 0\\
&  & \ddots &  & \\
0 & 0 & 0 & \cdots & \lambda
\end{pmatrix}
\qquad\qquad(d) \quad%
\begin{pmatrix}
R & I & {\mathbf{0}} & \cdots & {\mathbf{0}}\\
{\mathbf{0}} & R & I & \cdots & {\mathbf{0}}\\
&  & \ddots &  & \\
{\mathbf{0}} & {\mathbf{0}} & {\mathbf{0}} & \cdots & R
\end{pmatrix}
\]
\vskip 2mm

\noindent where $R$ is a rotation matrix of the form $%
\begin{pmatrix}
\cos\theta & -\sin\theta\\
\sin\theta & \cos\theta
\end{pmatrix}
$, ${\mathbf{0}}$ denotes the $2\times2$ zero matrix, and $I$ denotes the
$2\times2$ identity matrix. Let $V_{W} = \phi^{-1}(U_{E})$ and $V^{\prime} =
\phi^{-1}(U^{\prime})$. \vskip 2mm

Case 1. $L\,:\,W\rightarrow W$ is an isometry. This is equivalent to saying
that each Jordan block of $L$ is of type (a) or (b). The fact that $|\det
L|=1$, and regarding $L$ as acting on the unit sphere in $W$, implies that
$L(\overline{V_{W}})\subset{V_{W}}$ is not possible unless $V_{W}=E$, which in
turn implies that $f|_{E}(\overline{U_{E}})\subset U_{E}$ is not possible
unless $U_{E}=E$. Therefore, by equation (\ref{restriction-equation}) we have
$U_{E}=E$ and similarly $U_{E}^{\prime}=E$, which implies that $U\cap
U^{\prime}\neq\emptyset$, contradicting what was stated above. \vskip2mm

Case 2. There is at least one Jordan block in $L_{E}$ of the form (c) or (d).
Define the \textit{size} of an $m\times m$ Jordan block $B$ as $m$ if $B$ is
of type (c) and $m/2$ if $B$ is of type (d). Let $s$ be the maximum of the
sizes of all (c) and (d) type Jordan blocks. Let $\widehat{W}$ be the subspace
of ${\mathbb{R}}^{n+1}$ consisting of all points $(x_{0}x_{1},\dots,x_{n})$ in
homogeneous coordinates with $x_{i}=0$ for all $i$ not corresponding to the
first row of a Jordan block of type (c) and size $s$ or to the first two rows
of a Jordan block of type (d) and size $s$. Let $\widehat{E}$ be the
projective subspace corresponding to $\widehat{W}$. If $x\in U$, then it is
routine to check, by iterating the Jordan canonical form and scaling so that
the maximum modulus of an eigenvalue is 1, that $A(x)\subset\widehat{E}$.
Similarly, if $x^{\prime}\in U^{\prime}$, then $A(x^{\prime})\subset
\widehat{E}$. Therefore $U_{\widehat{E}}:=U\cap\widehat{E}\neq\emptyset$ and
$U_{\widehat{E}}^{\prime}:=U^{\prime}\cap\widehat{E}\neq\emptyset$. As done
above for $E$, if $f|_{\widehat{E}}$ denotes the restriction of $f$ to
$\widehat{E}$, then $f|_{\widehat{E}}(\overline{U_{\widehat{E}}})\subset
U_{\widehat{E}},$ and $f|_{\widehat{E}}(\overline{U_{\widehat{E}}^{\prime}%
})\subset U_{\widehat{E}}^{\prime}.$ But $\widehat{W}$ is invariant under $L$
and, if $\widehat{L}$ is the restriction of $L$ to $\widehat{W}$, then
$\widehat{L}$ is an isometry. We now arrive at a contradiction exactly as was
done in Case 1.
\end{proof}

\section{\label{dual}Duals and Adjoints}

Recall that $d_{\mathbb{P}}(\cdot,\cdot)$ is the metric on ${\mathbb{P}}^{n}$
defined in Section \ref{projsec}. The \textit{hyperplane orthogonal to}
$p\in\mathbb{P}$ is defined and denoted by
\[
p^{\bot}=\{q\in\mathbb{P}^{n}\,:\,q\bot p\}\text{.}%
\]

If $\left(  \mathbb{X},d_{\mathbb{X}}\right)  $ denotes a compact metric space
$\mathbb{X}$ with metric $d_{\mathbb{X}}$, then $\left(  \mathbb{H}%
(\mathbb{X)},h_{\mathbb{X}}\right)  $ denotes the corresponding compact metric
space that consists of the nonempty compact subsets of $\mathbb{X}$ with the
Hausdorff metric $h_{\mathbb{X}}$ derived from $d_{\mathbb{X}}$, defined by
\[
h_{\mathbb{X}}(B,C)=\max\,\{\sup_{b\in B}\inf_{c\in C}d_{\mathbb{X}%
}(b,c),\,\sup_{c\in C}\inf_{b\in B}d_{\mathbb{X}}(b,c)\}
\]
for all $B,C\in\mathbb{H}.$\ It is a standard result that if $\mathcal{F}%
=\left(  \mathbb{X};f_{1},f_{2},...,f_{M}\right)  $\ is a contractive IFS,
then $\mathcal{F}:\mathbb{H}(\mathbb{X)\rightarrow H}(\mathbb{X)}$ is a
contraction with respect to the Hausdorff metric.

\begin{definition}
The \textbf{dual space} $\widehat{{\mathbb{P}^{n}}}$ of {$\mathbb{P}^{n}$} is
the set of all hyperplanes of {$\mathbb{P}^{n}$}, equivalently
$\widehat{{\mathbb{P}^{n}}}=\{p^{\bot}:p\in\mathbb{P}^{n}\}$. The dual space
is endowed with a metric $d_{\widehat{{\mathbb{P}}}}$ defined by
\[
d_{\widehat{{\mathbb{P}}}}(p^{\bot},q^{\bot})=d_{{\mathbb{P}}}(p,q)
\]
for all $p^{\bot},q^{\bot}\in{\mathbb{\widehat{P}}}$. The map $\mathcal{D}%
\,:\,${$\mathbb{P}^{n}$}$\rightarrow\widehat{{\mathbb{P}^{n}}}$ defined by
\[
\mathcal{D}\left(  p\right)  =p^{\bot}%
\]
is called the \textbf{duality map}. The duality map can be extended to a map
$\mathcal{D}\,:\,{\mathbb{H}}(${$\mathbb{P}^{n}$}$)\rightarrow{\mathbb{H}%
}(\widehat{{\mathbb{P}^{n}}})$ between compact subsets of {$\mathbb{P}^{n}$}
and $\widehat{{\mathbb{P}^{n}}}$ in the usual way.
\end{definition}

Given a projective transformation $f$ and any matrix $L_{f}$ representing it,
the matrix $L_{f^{-1}}:=L_{f}^{-1}$ represents the projective transformation
$f^{-1}$ that is the inverse of $f$. In a similar fashion, define the
\textit{adjoint} $f^{t}$ and {the adjoint inverse} transformation $f^{-t}$ as
the projective transformations represented by the matrices
\[
L_{f^{t}}:=L_{f}^{t}\qquad\text{ and }\qquad L_{f^{-t}}:=(L_{f}^{-1}%
)^{t}=(L_{f}^{t})^{-1},
\]
respectively, where $t$ denotes the transpose matrix. It is easy to check that
the adjoint and adjoint inverse are well defined. For a projective IFS
$\mathcal{F}$, the following related iterated function systems will be used in
this section.

\begin{enumerate}
\item The \textit{adjoint of the projective IFS} $\mathcal{F}$ is denoted by
$\mathcal{F}^{t}$ and defined to be
\[
\mathcal{F}^{t}=\left(  \mathbb{P}^{n};f_{1}^{t},f_{2}^{t},...,f_{M}%
^{t}\right)  .
\]

\item The \textit{inverse of the projective IFS} $\mathcal{F}$ is the
projective IFS
\[
\mathcal{F}^{-1}=\left(  \mathbb{P}^{n};f_{1}^{-1},f_{2}^{-1},...,f_{M}%
^{-1}\right)  .
\]

\item If $\mathcal{F}=\left(  \mathbb{P}^{n};f_{1},f_{2},...,f_{M}\right)  $
is a projective IFS then the \textit{corresponding hyperplane IFS} is
\[
\widehat{\mathcal{F}}=(\widehat{{\mathbb{P}^{n}}};f_{1},f_{2},...,f_{M}),
\]
where $f_{m}:\widehat{{\mathbb{P}^{n}}}\rightarrow\widehat{{\mathbb{P}^{n}}}$
is defined by $f_{m}(H)=\{f_{m}(q)\,|\,q\in H\}$. Notice that, whereas
$\mathcal{F}$ is associated with the compact metric space $(\mathbb{P}%
^{n},d_{\mathbb{P}}),$ the hyperplane IFS $\widehat{\mathcal{F}}$ is
associated with the compact metric space $(\widehat{{\mathbb{P}^{n}}%
},d_{\widehat{{\mathbb{P}}}})$. \vskip2mm

\item The \textit{corresponding inverse hyperplane IFS} is
\[
\widehat{\mathcal{F}^{-1}}=(\widehat{{\mathbb{P}^{n}}};f_{1}^{-1},f_{2}%
^{-1},...,f_{M}^{-1}),
\]
where $f_{m}^{-1}:\widehat{{\mathbb{P}^{n}}}\rightarrow\widehat{{\mathbb{P}%
^{n}}}$ is defined by $f_{m}^{-1}(H)=\{f_{m}^{-1}(q)\,|\,q\in H\}$.
\end{enumerate}

\begin{proposition}
\label{compactH} The duality map $\mathcal{D}$ is a continuous, bijective,
inclusion preserving isometry between compact metric spaces $\left(
\mathbb{P}^{n},d_{{\mathbb{P}}}\right)  $ and $\left(  \widehat{{\mathbb{P}%
^{n}}},d_{\widehat{{\mathbb{P}}}}\right)  $ and also a continuous, bijective,
inclusion preserving isometry between $\left(  {\mathbb{H}}(\mathbb{P}%
^{n}),h_{{\mathbb{P}}}\right)  $ and $\left(  {\mathbb{H}}%
(\widehat{{\mathbb{P}^{n}}}),h_{\widehat{{\mathbb{P}}}}\right)  $. Moreover,
the following diagrams commute for any projective transformation $f$ and any
projective IFS $\mathcal{F}$:
\[%
\begin{array}
[c]{ccc}
& \mathcal{D} & \\
{\mathbb{P}^{n}} & \rightarrow & \widehat{{\mathbb{P}^{n}}}\\
f^{t}\downarrow &  & \downarrow f^{-1}\\
{\mathbb{P}^{n}} & \rightarrow & \widehat{{\mathbb{P}^{n}}}\\
& \mathcal{D} &
\end{array}
\qquad\qquad%
\begin{array}
[c]{ccc}
& \mathcal{D} & \\
{\mathbb{H}}({\mathbb{P}^{n}}) & \rightarrow & {\mathbb{H}}%
(\widehat{{\mathbb{P}^{n}}})\\
\mathcal{F}^{t}\downarrow &  & \downarrow\widehat{\mathcal{F}^{-1}}\\
{\mathbb{H}}({\mathbb{P}^{n}}) & \rightarrow & {\mathbb{H}}%
(\widehat{{\mathbb{P}^{n}}}).\\
& \mathcal{D} &
\end{array}
\]

\end{proposition}

\begin{proof}
Clearly $\mathcal{D}$ maps {$\mathbb{P}^{n}$} bijectively onto
$\widehat{{\mathbb{P}^{n}}}$ and ${\mathbb{H}}(${$\mathbb{P}^{n}$}$)$
bijectively onto ${\mathbb{H}}(\widehat{{\mathbb{P}^{n}}})$. The continuity of
$\mathcal{D}$ and the inclusion preserving property are also clear. The
definition of $d_{\widehat{{\mathbb{P}}}}$ in terms of $d_{\mathbb{P}}$
implies that $\mathcal{D}$ is an isometry from {$\mathbb{P}^{n}$} onto
$\widehat{{\mathbb{P}^{n}}}$. The definition of $h_{\widehat{{\mathbb{P}}}}$
in terms of $d_{\widehat{{\mathbb{P}}}}$ and the definition of $h_{\mathbb{P}%
^{n}}$ in terms of $d_{\mathbb{P}}$ implies that $\mathcal{D}$ is an isometry
from $\mathbb{H}(\mathbb{P}^{n})$ onto $\mathbb{H}(\widehat{{\mathbb{P}^{n}}%
})$. The compactness of $\left(  \mathbb{P}^{n},d_{\mathbb{P}}\right)  $
implies that $(\widehat{{\mathbb{P}^{n}}},d_{\widehat{{\mathbb{P}}}})$ is a
compact metric space.

To verify that the diagrams commute it is sufficient to show that, for all
$x\in${$\mathbb{P}^{n}$} and any projective transformation $f$, we have
$L_{f}^{-1}(x^{\bot})=[L_{f}^{t}(x)]^{\bot}$. But, using the ordinary
Euclidean inner product,
\[%
\begin{array}
[c]{rl}%
L_{f}^{-1}(x^{\bot})\;= & \{L_{f}^{-1}y\,:\,\langle x,y\rangle
=0\}=\{z\,:\,\langle x,L_{f}z\rangle=0\}\\
= & \{z\,:\,\langle L_{f}^{t}x,z\rangle=0\}=[L_{f}^{t}(x)]^{\bot}.
\end{array}
\]

\end{proof}

Let ${\mathbb{S}}(${$\mathbb{P}^{n}$}$)$ denote the set of all subsets of
{$\mathbb{P}^{n}$} (including the empty set). \vskip 2mm

\begin{definition}
The \textbf{complementary dual} of a set $X\subset\mathbb{P}^{n}$ is
\[
X^{\ast}=\{q\in\mathbb{P}^{n}:q\bot x\,\text{ for no }x\in X\},
\]

\end{definition}

\vskip 2mm

For an IFS $\mathcal{F}$ define the operator $\overline{\mathcal{F}%
}:{\mathbb{S}}(${$\mathbb{P}^{n}$}$)\rightarrow\mathbb{S}(${$\mathbb{P}^{n}$%
}$)$ by
\[
\overline{\mathcal{F}}(X)=\bigcap_{f\in\mathcal{F}}\,f^{-t}(X),
\]
for any $X\in\mathbb{S}(${$\mathbb{P}^{n}$}$)$.

\begin{proposition}
\label{symbolic} The map $^{\ast} : \, \mathbb{S}(\mathbb{P}^{n})
\rightarrow\mathbb{S}(\mathbb{P}^{n})$ is an inclusion reversing function with
these properties:

1. The following diagram commutes
\[%
\begin{array}
[c]{ccc}
& {\ast} & \\
\mathbb{S} (\mathbb{P}^{n}) & \rightarrow & \mathbb{S}({\mathbb{P}^{n}})\\
\mathcal{F}\downarrow &  & \downarrow\overline{\mathcal{F}}\\
\mathbb{S}({\mathbb{P}^{n}}) & \rightarrow & \mathbb{S}({\mathbb{P}^{n}}).\\
& {\ast} &
\end{array}
\]

2. If $\mathcal{F}(X)\subset Y$, then $\mathcal{F}^{t}(Y^{\ast})\subset
X^{\ast}$.

3. If $X$ is open, then $X^{\ast}$ is closed. If $X$ is closed, then $X^{\ast
}$ is open.

4. $\overline{\overline{X}^{\ast}}\subset X^{\ast}$.
\end{proposition}

\begin{proof}
The fact that the diagrams commute is easy to verify. Since the other
assertions are also easy to check, we prove only statement (3). Since $^{\ast
}$ is inclusion reversing, $\mathcal{F}(X)\subset Y$ implies that $Y^{\ast
}\subset\lbrack\mathcal{F}(X)]^{\ast} =\overline{\mathcal{F}}(X^{\ast})$, the
equality coming from the commuting diagram. The definition of $\overline
{\mathcal{F}}$ then yields $\mathcal{F}^{t}(Y^{\ast})\subset X^{\ast}$.
\end{proof}

\begin{proposition}
\label{inclusion}If $\mathcal{F}$ is a projective IFS, $U\subset\mathbb{P}%
^{n}$ is open, and $\mathcal{F}(\overline{U})\subset U,$ then $V=\overline
{U}^{\ast}$ is open and $\mathcal{F}^{t}(\overline{V})\subset V$.
\end{proposition}

\begin{proof}
From statement (3) of Proposition \ref{symbolic} it follows that $V$ is open.
From $\mathcal{F}(\overline{U})\subset U$ and from statement (2) of
Proposition \ref{symbolic} it follows that $\mathcal{F}^{t}(U^{\ast}%
)\subset\overline{U}^{\ast}$. By statement (4) we have $\mathcal{F}%
^{t}(\overline{V})=\mathcal{F}^{t}(\overline{\overline{U}^{\ast}}%
)\subset\mathcal{F}^{t}(U^{\ast})\subset\overline{U}^{\ast}=V$.
\end{proof}

\begin{lemma}
\label{(1)equiv(5)} \textbf{[Theorem \ref{main} (1)} $\Leftrightarrow
$\textbf{(5)]:} A projective IFS $\mathcal{F}$ has an attractor $A$ that
avoids a hyperplane if and only if $\mathcal{F}^{t}$ has an attractor $A^{t}$
that avoids a hyperplane.
\end{lemma}

\begin{proof}
Suppose statement $(1)$ of Theorem \ref{main} is true. By statement $(2)$ of
Theorem \ref{main} there is a nonempty open set $U$ and a hyperplane $H$ such
that $\mathcal{F}(\overline{U})\subset U$ and $H\cap\overline{U}=\emptyset$.
By Proposition \ref{inclusion} we have $\mathcal{F}^{t}(\overline{V})\subset
V$ where $V=\overline{U}^{\ast}$ is open. Moreover, there is a hyperplane
$H^{t}$ such that $H^{t}\cap\overline{V}=\emptyset$: simply choose
$H^{t}=a^{\bot}$ for any $a\in A\subset U$, where $A$ is the attractor of
$\mathcal{F}$. By the definition of the dual complement, $a^{\bot}\cap
U^{\ast}=\emptyset$ which, by statement (4) of Proposition \ref{symbolic},
implies that $a^{\bot}\cap\overline{V}=a^{\bot}\cap\overline{\overline
{U}^{\ast}}=\emptyset$. So, as long as $V\neq\emptyset$, $\mathcal{F}^{t}$
also satisfies statement $(2)$ of Theorem \ref{main}. In this case it follows
that statement $\left(  1\right)  $ of Theorem \ref{main} is true for
$\mathcal{F}^{t}$, and hence statement $(5)$ is true.

We show that $V\neq\emptyset$ by way of contradiction. If $V=\emptyset$, then
by the definition of the dual complement, every $y\in\mathbb{P}^{n}$ is
orthogonal to some point in $\overline{U},$ i.e.%
\[
{\overline{U}}^{\bot}:=\{y\,:\,y\perp x\;\;\mbox{for some}\;\;x\in\overline
{U}\}={\mathbb{P}^{n}}.
\]
On the other hand, since $\overline{U}$ avoids some hyperplane $y^{\bot}$, we
arrive at the contradiction $y\notin\overline{U}^{\bot}$.

The converse in Lemma \ref{(1)equiv(5)} is immediate because $\left(
\mathcal{F}^{t}\right)  ^{t}=\mathcal{F}$.
\end{proof}

\begin{definition}
A set $\mathcal{A}\subset\widehat{{\mathbb{P}^{n}}}$ is called a
\textbf{hyperplane} \textbf{attractor} of the projective IFS $\mathcal{F}$ if
it is an attractor of the IFS $\widehat{\mathcal{F}}.$ A set $R\subset
\mathbb{P}^{n}$ is said to be a \textbf{repeller} of the projective IFS
$\mathcal{F}$ if $R$ is an attractor of the inverse IFS $\mathcal{F}^{-1}.$ A
set $\mathcal{R}\subset\widehat{{\mathbb{P}^{n}}}$ is said to be a
\textbf{hyperplane} \textbf{repeller} of the projective IFS $\mathcal{F}$ if
it is a hyperplane attractor of the inverse hyperplane IFS
$\widehat{\mathcal{F}^{-1}}$.
\end{definition}

\begin{proposition}
\label{dualIFSlemma} The compact set $A\subset\mathbb{P}^{n}$ is an attractor
of the projective IFS $\mathcal{F}^{t}$ that avoids a hyperplane if and only
if $\mathcal{D}(A)$ is a hyperplane repeller of $\mathcal{F}$ that avoids a point.
\end{proposition}

\begin{proof}
Concerning the first of the two conditions in the definition of an attractor,
we have from the commuting diagram in Proposition \ref{compactH} that
$\mathcal{F}^{t}(A)=A$ if and only if $\widehat{\mathcal{F}^{-1}}%
(\mathcal{D}(A))=\mathcal{D}(\mathcal{F}^{t}(A))=\mathcal{D}(A)$.

Concerning the second of the two conditions in the definition of an attractor,
let $B$ be an arbitrary subset contained in the basin of attraction $U$ of
$\mathcal{F}^{t}$. With respect to the Hausdorff metric, $\lim_{k\rightarrow
\infty}(\mathcal{F}^{t})^{k}(B)=A$ if and only if
\[
\lim_{k\rightarrow\infty}\widehat{\mathcal{F}^{-1}}^{k}(\mathcal{D}%
(B))=\lim_{k\rightarrow\infty}\mathcal{D}((\mathcal{F}^{t})^{k}%
(B))=\mathcal{D}(\lim_{k\rightarrow\infty}(\mathcal{F}^{t})^{k}%
(B))=\mathcal{D}(A).
\]
Also, the attractor $\mathcal{D}(A)$ of $\widehat{\mathcal{F}^{-1}}$ avoids
the point $p$ if and only if the attractor $A$ of $\mathcal{F}^{t}$ avoids the
hyperplane $p^{\bot}$.
\end{proof}

\begin{lemma}
\label{newlemma} Let $f:\mathbb{P}^{n}\rightarrow\mathbb{P}^{n}$ be a
projective transformation with attractive fixed point $p_{f}$ and
corresponding invariant hyperplane $H_{f}$.\ If $f^{-1}:\widehat{\mathbb{P}%
^{n}}\rightarrow\widehat{\mathbb{P}^{n}}$ has an attractive fixed point
$\widehat{H_{f}}$, then $\widehat{H_{f}}=H_{f}$.
\end{lemma}

\begin{proof}
There is some basis with respect to which $f$ has matrix $%
\begin{pmatrix}
L & 0\\
0 & 1
\end{pmatrix}
$. If $f$ is represented by matrix $L_{f}$ with respect to the standard basis,
then there is an invertible matrix $M$ such that
\[
L_{f}=M
\begin{pmatrix}
L & 0\\
0 & 1
\end{pmatrix}
M^{-1},
\]
where $L$ is a non-singular $n\times n$ matrix whose eigenvalues $\lambda$
satisfy $|\lambda|<1$. Then
\[
L_{f}^{-1}=M
\begin{pmatrix}
L^{-1} & 0\\
0 & 1
\end{pmatrix}
M^{-1}\text{ and }L_{f}^{t}=M^{-t}
\begin{pmatrix}
L^{t} & 0\\
0 & 1
\end{pmatrix}
M^{t}.
\]
If $x=(0,0,...,0,1),$ then by Proposition \ref{compactH}
\[
\widehat{H_{f}}=\left(  M^{-t}x\right)  ^{\bot}=M(x^{\bot})=H_{f}.
\]

\end{proof}

\begin{proposition}
\label{open} If $\mathcal{F}$ is a projective IFS and $U$ is an open set such
that $\overline{U}$ avoids a hyperplane and $\mathcal{F}(\overline{U})\subset
U$, then $\mathcal{F}$ has an attractor $A$ and $U$ is contained in the basin
of attraction of $A$.
\end{proposition}

\begin{proof}
We begin by noting that $\mathcal{F}(\overline{U})\subset U$ implies that
$\left\{  \mathcal{F}^{k}(\overline{U})\right\}  _{k=1}^{\infty}$ is a nested
sequence of nonempty compact sets. So
\[
\widetilde{A}:=\bigcap_{k=1}^{\infty}\mathcal{F}^{k}(\overline{U})
\]
is also a nonempty compact set. Using the continuity of $\mathcal{F}%
:\mathbb{H(P}^{n}\mathbb{)\rightarrow}$ $\mathbb{H(P}^{n}\mathbb{)}$, we have
$\mathcal{F}(\widetilde{A})=\widetilde{A}$.

If $B\in\mathbb{H(P}^{n}\mathbb{)}$ is such that $B\subset U,$ then, given any
$\varepsilon>0$, there is a positive integer $K :=K(\varepsilon)$ such that
$\mathcal{F}^{K(\varepsilon)}(B)\subset\widetilde{A}_{\varepsilon}$, the set
$\widetilde{A}$ dilated by an open ball of radius $\varepsilon.$

In the next paragraph we are going to show that, for sufficiently small
$\varepsilon>0,$ there is a metric on $\widetilde{A}_{\varepsilon}$ such that
$\mathcal{F}$ is contractive on $\widetilde{A}_{\varepsilon}.$ For now, assume
that $\mathcal{F}$ is contractive on $\widetilde{A}_{\varepsilon}$. This
implies, by Theorems \ref{main} and \ref{uniquethm}, that $\mathcal{F}$ has a
unique attractor $A$ and it is contained in $\widetilde{A}_{\varepsilon}$. We
now show that $A=\widetilde{A}$. \ That $\mathcal{F}$ is contractive on
$\widetilde{A}_{\varepsilon}$ implies that $\mathcal{F},$ considered as a
mapping on $\mathbb{H}\left(  \widetilde{A}_{\varepsilon}\right)  ,$ is a
contraction with respect to the Hausdorff metric. By the contraction mapping
theorem, $\mathcal{F}$ has a unique fixed point, so $A=\widetilde{A}.$ By
choosing $\varepsilon$ small enough that $\widetilde{A}_{\varepsilon
}=A_{\varepsilon}$ lies in the basin of attraction of $A,$ the fact that
$\mathcal{F}^{K}(B)\subset\widetilde{A}_{\varepsilon}$ implies that
$\lim_{k\rightarrow\infty}\mathcal{F}^{k}(B)=A$. Hence $U$ lies in the basin
of attraction of $A$, which concludes the proof of Proposition \ref{open}.

To prove that $\mathcal{F}$ is contractive on $\widetilde{A}_{\varepsilon}$
for sufficiently small $\varepsilon>0,$ we follow the steps in the
construction of the metric in statement (4) of Theorem \ref{main}, starting
from the proof of Lemma \ref{(2)implies(3)A}. As in the proof of Lemma
\ref{(2)implies(3)A}, let $U=\cup_{\alpha}U_{\alpha}$, where the $U_{\alpha}$
are the connected components of $U$. Let $\left\{  O_{i}\right\}  $ be the set
of $U_{\alpha}$ that have nonempty intersection with $\tilde{A}$. Since
$\widetilde{A}$ is compact and nonempty, we must have
\[
(\widetilde{A}_{\varepsilon})\subset\bigcup_{i}O_{i}%
\]
for all $\varepsilon$ sufficiently small. We now follow the steps in the proof
of Lemma \ref{(2)implies(3)A}, Lemma \ref{(2)implies(3)B}, up to and including
Lemma \ref{(3)implies(4)lemma}, to construct a metric on a finite set of
convex bodies $\left\{  C_{i}\right\}  $ such that $\cup_{i}O_{i}\subset
\cup_{i}C_{i}$ and such that $\mathcal{F}$ is contractive on $\cup_{i}C_{i}$.
Note that the metric is constructed on a set containing $\cup_{i}O_{i},$ which
in turn contains $\widetilde{A}_{\varepsilon}$. This completes the proof.
\end{proof}

We can now prove Theorem \ref{adjointhm}. \vskip2mm

\begin{proof}
[\textbf{Proof of Theorem 3}]We prove the first statement of the theorem. The
proof of the second statement is identical with $\mathcal{F}$ replaced by
$\mathcal{F}^{-1}$.

Assume that projective IFS $\mathcal{F}$ has an attractor that avoids a
hyperplane. By statement (4) of Theorem \ref{main}, the IFS $\mathcal{F}^{t}$
has an attractor that avoids a hyperplane. Then, according to Proposition
\ref{dualIFSlemma}, $\widehat{\mathcal{F}^{-1}}$ has an attractor that avoids
a point. By definition of hyperplane repeller, $\mathcal{F}$ has a hyperplane
repeller that avoids a point.

Concerning the basin of attraction, let $R$ denote the union of the
hyperplanes in $\mathcal{R}$ and let $Q=\mathbb{P}^{n}\smallsetminus R$. We
must show that $Q=O,$ where $O$ is the basin of attraction of the attractor
$A$ of $\mathcal{F}$.

First we show that $O\subset Q$, i.e. $O\cap R=\emptyset$. Consider any
$f:\mathbb{P}^{n}\rightarrow\mathbb{P}^{n}$ with $f\in\mathcal{F}$ and
$f^{-1}:\widehat{\mathbb{P}^{n}}\rightarrow\widehat{\mathbb{P}^{n}}$. Since we
have already shown that $\widehat{\mathcal{F}^{-1}}$ has an attractor, it
satisfies all statements of Theorem \ref{main}. \ It then follows, exactly as
in the proof of Lemma \ref{4implies1}, that $f^{-1}:\widehat{\mathbb{P}^{n}%
}\rightarrow\widehat{\mathbb{P}^{n}}$ has an attractive fixed point, a
hyperplane $\widehat{H_{f}}\in\mathcal{R}\subset\widehat{\mathbb{P}^{n}}$. Let%
\[
\mathcal{B}=\overline{\bigcup_{k=1}^{\infty}\bigcup_{f\in\mathcal{F}}\left(
\widehat{\mathcal{F}^{-1}}\right)  ^{k}(\widehat{H_{f}})}\subset
\widehat{\mathbb{P}^{n}}\qquad\text{ and }\qquad B=\bigcup_{H\in\mathcal{B}%
}H.
\]
The fact that $\widehat{H_{f}}=H_{f}$ (Lemma \ref{newlemma}) and $H_{f}\cap
O=\emptyset$ for all $f\in\mathcal{F}$ implies that $O\cap B=\emptyset$. We
claim that $\mathcal{B}=\mathcal{R}$ and hence $B=R,$ which would complete the
proof that $\ O\cap R=\emptyset$. Concerning the claim, because $\mathcal{R}$
is the attractor of $\widehat{\mathcal{F}^{-1}},$ we have that%
\[
\mathcal{R}=\lim_{k\rightarrow\infty}\left(  \widehat{\mathcal{F}^{-1}%
}\right)  ^{k}\left(  \bigcup_{f\in\mathcal{F}}\widehat{H_{f}}\right)
\subset\mathcal{B}.
\]
Since $\widehat{H_{f}}\in\mathcal{R}$ for all $f$ $\in$ F, also $\mathcal{B}%
\subset\mathcal{R}$, which completes the proof of the claim.

Finally we show that $Q\subset O$. By statements $(2)$ and $\left(  5\right)
$ of Theorem \ref{main}, $\mathcal{F}^{t}$ has an attractor $A^{t}$ that
avoids a hyperplane. Consequently there is an open neighborhood $V$ of $A^{t}$
and a metric such that $\mathcal{F}^{t}$ is contractive on $\overline{V},$ and
$\overline{V}$ avoids a hyperplane. In particular $\mathcal{F}^{t}$ is a
contraction on $\mathbb{H}\left(  \overline{V}\right)  $ with respect to the
Hausdorff metric. Let $\lambda$ denote a contractivity factor for
$\mathcal{F}^{t}|_{\overline{V}}$. Let $\varepsilon>0$ be small enough that
the closed set $A^{t}_{\varepsilon}$ (the dilation of $A^{t}$ by a closed ball
of radius $\varepsilon$, namely the set of all points whose distance from
$A^{t}$ is less than or equal to $\varepsilon$) is contained in $V$. If
$h_{\mathbb{P}}(A^{t}_{\varepsilon},A^{t})=\varepsilon$, then
\begin{align*}
h_{\mathbb{P}}(\mathcal{F}^{t}(A^{t}_{\varepsilon}),A^{t})  &  =h_{\mathbb{P}%
}(\mathcal{F}^{t}(A^{t}_{\varepsilon}),\mathcal{F}^{t}(A^{t}))\\
&  \leq\lambda h_{\mathbb{P}}(A^{t}_{\varepsilon},A^{t}))=\lambda\varepsilon.
\end{align*}
It follows that $\mathcal{F}^{t}(A^{t}_{\varepsilon})\subset int(A^{t}%
_{\varepsilon})$ and from Proposition \ref{symbolic} (2,3) that
\[
\mathcal{F}(\overline{(A^{t}_{\varepsilon})^{\ast}})\subseteq\mathcal{F}%
(int((A^{t}_{\varepsilon})^{\ast}))\subset(A^{t}_{\varepsilon})^{\ast}.
\]
Let $Q_{\varepsilon}:=(A^{t}_{\varepsilon})^{\ast}$. It follows from
$\mathcal{F}(\overline{Q_{\varepsilon}})\subset Q_{\varepsilon}$ and
Proposition \ref{open} that $Q_{\varepsilon}\subset O$. Let $\mathcal{R}%
_{\varepsilon}=\mathcal{D}\left(  A^{t}_{\varepsilon}\right)  $ and let
$R_{\varepsilon}\subset\mathbb{P}^{n}$ be the union of the hyperplanes in
$\mathcal{R}_{\varepsilon}.$ By Proposition \ref{dualIFSlemma} and the
definition of the dual complement, $Q_{\varepsilon}=\mathbb{P}^{n}\backslash
R_{\varepsilon}$ and $Q=\mathbb{P}^{n}\backslash R$. Since $Q_{\varepsilon
}\subset O$ it follows that $R_{\varepsilon}\subset\mathbb{P}^{n}\backslash
O$. Since $\mathcal{D}$ is continuous (Proposition \ref{compactH}) and
$A^{t}_{\varepsilon}\rightarrow A^{t}$ as $\varepsilon\rightarrow0$, it
follows that $\mathcal{R}_{\varepsilon}=\mathcal{D}\left(  A^{t}_{\varepsilon
}\right)  \rightarrow\mathcal{D(}A^{t})=\mathcal{R}$. $\ $Consequently
$R\subset\mathbb{P}^{n}\backslash O,$ and therefore $Q=\mathbb{P}%
^{n}\backslash R\subset O$.
\end{proof}

\section{\label{groupinvsec}Geometrical Properties of Attractors}

The Hausdorff dimension of the attractor of a projective IFS is invariant
under the projective group $PGL(n+1,\mathbb{R})$. This is so because any
projective transformation is bi-Lipshitz with respect to $d_{\mathbb{P}},$
that is, if $f:\mathbb{P}^{n}\rightarrow\mathbb{P}^{n}$ is a projective
transformation, then there exist two constants $0<\lambda_{1}<\lambda
_{2}<\infty$ such that
\[
\lambda_{1}d_{\mathbb{P}}(x,y)\leq d_{\mathbb{P}}(f(x),f(y))\leq\lambda
_{2}d_{\mathbb{P}}(x,y)\text{.}%
\]
We omit the proof as it is a straightforward geometrical estimate.

The main focus of this section is another type of invariant that depends both
on the attractor and on a corresponding hyperplane repeller. It is a type of
Conley index and is relevant to the study of parameter dependent families of
projective IFSs and the question of when there exists a continuous family of
IFS's whose attractors interpolate a given set of projective attractors, as
discussed in Example 4 in Section \ref{examplesec}. Ongoing studies suggest
that this index has stability properties with respect to small perturbations
and that there does not exist a family of projective IFSs whose attractors
continuously interpolate between attractors with different indices.

\begin{definition}
\label{indexdef} Let $\mathcal{F}$ be a projective IFS with attractor $A$ that
avoids a hyperplane and let $R$ denote the union of the hyperplanes in the
hyperplane repeller of $\mathcal{F}$. The \textbf{index of} $\mathcal{F}$ is
\[
index(\mathcal{F})=\#\text{ connected components $O$ of }\mathbb{P}%
^{n}\backslash R\text{ such that }A\cap O\neq\emptyset\text{.}%
\]

\end{definition}

Namely, the index of a contractive projective IFS is the number of components
of the open set $\mathbb{P}^{n}\backslash R$ which have non-empty intersection
with its attractor. By statement (1) of Theorem 3, we know that
$index(\mathcal{F})$ will always equal a positive integer.

\begin{definition}
\label{indexdef2} Let $A$ denote a nonempty compact subset of $\mathbb{P}^{n}%
$, that avoids a hyperplane. If $\mathbb{F}_{A}$ denotes the collection of all
projective IFSs for which $A$ is an attractor, then the \textbf{index of} $A$
is defined by the rule%
\[
index(A)=\min_{\mathcal{F}\in\mathbb{F}_{A}}\{index(\mathcal{F})\}\text{.}%
\]
If the collection $\mathbb{F}_{A}$ is empty, then define $index(A)=0$.
\end{definition}

Note that an attractor $A$ not only has a multitude of projective IFSs
associated with it, but it may also have a multitude of repellers associated
with it. Clearly $index(A)$ is invariant under under $PGL(n+1,\mathbb{R}),$
the group of real projective transformations. The following lemma shows that,
for any positive integer, there exists a projective IFS $\mathcal{F}$ that has
that integer as index.

\begin{proposition}
\label{indexflemma} Let $\mathcal{F}=(\mathbb{P}^{1};f_{1},f_{2}
,f_{3},...,f_{M})$ be a projective IFS where, for each $m$, the projective
transformation $f_{m}$ is represented by the matrix
\[
L_{m}:=
\begin{pmatrix}
2m\lambda-2m+1 & 2m\left(  m-\frac{1}{2}\right)  -m\lambda\left(  2m-1\right)
\\
2\lambda-2 & 2m-\lambda\left(  2m-1\right)
\end{pmatrix}
\text{.}%
\]
For any integer $M>1$ and sufficiently large $\lambda$, the projective IFS has
$index(\mathcal{F})=M.$
\end{proposition}

\begin{proof}
Topologically, the projective line $\mathbb{P}^{1}$ is a circle. It is readily
verified that
\[
L_{m}=%
\begin{pmatrix}
\lambda m & m-\frac{1}{2}\\
\lambda & 1
\end{pmatrix}%
\begin{pmatrix}
m & m-\frac{1}{2}\\
1 & 1
\end{pmatrix}
^{-1}\text{,}%
\]
from which it can be easily checked that, for $\lambda$ is sufficiently large,
$f_{m}$ has attractive fixed point $x_{m}=%
\begin{pmatrix}
m\\
1
\end{pmatrix}
$ and repulsive fixed point $y_{m}=%
\begin{pmatrix}
m-\frac{1}{2}\\
1
\end{pmatrix}
$. In particular $L_{m}%
\begin{pmatrix}
m\\
1
\end{pmatrix}
=\frac{\lambda}{2}%
\begin{pmatrix}
m\\
1
\end{pmatrix}
$ and $L_{m}%
\begin{pmatrix}
m-\frac{1}{2}\\
1
\end{pmatrix}
=%
\begin{pmatrix}
m-\frac{1}{2}\\
1
\end{pmatrix}
$. Note that the points $x_{i},\,i=1,2,\dots,M,$ and $y_{i},\,i=1,2,\dots,M,$
interlace on the circle (projective line). Also, as $\lambda$ increases, the
attractive fixed points $x_{m}$ become increasingly attractive.

Let $I_{k}$ denote a very small interval that contains the attractive fixed
point $x_{k}$ of $f_{k}$, for $k=1,2,...,M.$ When $\lambda$ is sufficiently
large, $f_{m}(\cup I_{k})\subset I_{m}\subset\cup I_{k}.$ It follows that the
attractor of $\mathcal{F}$ is a Cantor set contained in $\cup I_{k}$.
Similarly, the hyperplane repeller of $\mathcal{F}$ consists of another Cantor
set that lies very close to the set of points $\{k-0.5:k=1,2,...,M\}$. It
follows that $index(\mathcal{F})=M$.
\end{proof}

Another example is illustrated in Figure \ref{fourcpts}. In this case the
underlying space has dimension two and the IFS $\mathcal{F}$ has
$index(\mathcal{F})=4$.%
\begin{figure}[ptb]%
\centering
\includegraphics[
natheight=13.652900in,
natwidth=13.652900in,
height=2.3495in,
width=2.3495in
]%
{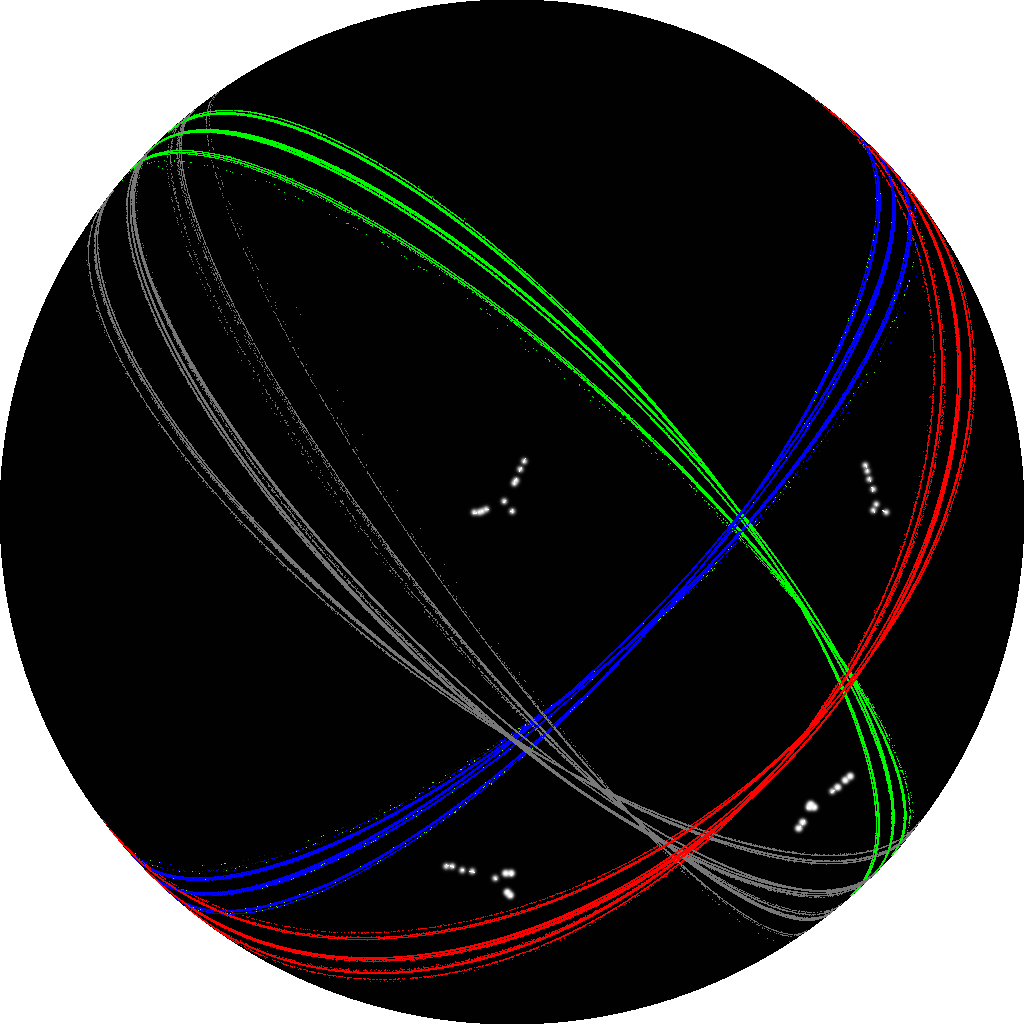}%
\caption{A projective IFS with index equal to four. The attractor is sketched
in white, while the union of the hyperplanes in the hyperplane repeller is
indicated in red, blue, green and gray.}%
\label{fourcpts}%
\end{figure}

The previous result shows that the index of a contractive IFS can be any
positive integer. It does not state that the same is true for the index of an
attractor. The following Theorem \ref{indexthm} shows that the index of an
attractor is a nontrivial invariant in that it is not always the case that
$index(A)=1$. To prove it we need the following definition and result.

\begin{definition}
A set $C\subset\mathbb{P}^{n}$ is called a \textbf{Cantor set} if it is the
attractor of a contractive IFS $(\mathbb{P}^{n};f_{1},f_{2},...,f_{M}),$
$M\geq2$, such that each point of $C$ corresponds to an unique string
$\sigma=\sigma_{1}\sigma_{2}\dots\in\{1,2,...,M\}^{\infty}$ such that
\begin{equation}
x=\varphi_{F}(\sigma)=\lim_{k\rightarrow\infty}f_{_{\sigma_{1}}}\circ
f_{\sigma_{2}}\circ\cdots\circ f_{\sigma_{k}}(x_{0}),
\label{bijection-equation}%
\end{equation}
where $x_{0}$ is any point in $C$.
\end{definition}

\begin{lemma}
\label{final-lemma} Let $\mathcal{F}$ $=(\mathcal{P}^{n};f_{1},f_{2}%
,...,f_{M})$ be a projective IFS whose attractor is a Cantor set $C$. Let the
projective IFS
\[
\mathcal{G}=(\mathbb{P}^{n};f_{\omega_{1}},f_{\omega_{2}},\dots,f_{\omega_{L}%
})
\]
have the same attractor $C,$ where each $f_{\omega_{l}}$ is a finite
composition of functions in $\mathcal{F}$, i.e.
\[
f_{\omega_{l}}=f_{\sigma_{1}^{l}}\circ f_{\sigma_{2}^{l}}\circ\cdots\circ
f_{\sigma_{j_{l}}^{l}}%
\]
in the obvious notation. Then $\mathcal{F}$ and $\mathcal{G}$ have the same
hyperplane repeller and $index(\mathcal{F})=index(\mathcal{G})$.
\end{lemma}

\begin{proof}
We must show that $\mathcal{R}_{\mathcal{G}}=\mathcal{R}_{\mathcal{F}}$, where
$\mathcal{R}_{\mathcal{F}}$ is the hyperplane repeller of $\mathcal{F}$ and
$\mathcal{R}_{\mathcal{G}}$ is the the hyperplane repeller of $\mathcal{G}$.
Let $\sigma=\sigma_{1}\sigma_{2}\cdots$ and $\omega_{l_{1}}\omega_{l_{2}%
}\cdots$ be strings of symbols in $\{1,2,...,M\}^{\infty}$ and $\{\omega
_{1},\omega_{2},\dots,\omega_{L}\}^{\infty}$, respectively. Define
\[
\psi:\{\omega_{1},\omega_{2},\dots,\omega_{L}\}^{\infty}\rightarrow
\{1,2,...,M\}^{\infty}%
\]
by
\[
\psi(\omega_{l_{1}}\omega_{l_{2}}\cdots)=\zeta(\omega_{l_{1}})\,\zeta
(\omega_{l_{2}})\cdots\text{ \textit{where} }\zeta(\omega_{l})=\sigma_{1}%
^{l}\,\sigma_{2}^{l}\cdots\sigma_{j_{l}}^{l}.
\]
We claim that $\psi$ is surjective. It is well known that the mapping
$\varphi_{F}:\{1,2,...,M\}^{\infty}\rightarrow C$ in equation
(\ref{bijection-equation}) is a continuous bijection, see for example
\cite[Chapter 4]{FE1}. Let $\sigma=\sigma_{1}\sigma_{2}\cdots\in
\{1,2,...,M\}^{\infty}$ and let $x=\lim_{k\rightarrow\infty}f_{\sigma_{_{1}}%
}\circ f_{\sigma_{2}}\circ\cdots\circ f_{\sigma_{k}}(x_{0})$. Since $C$ is
also the attractor of $\mathcal{G}$ it is likewise true that there is at least
one string $\omega=\omega_{l_{1}}\omega_{l_{2}}\cdots\in\{\omega_{1}%
,\omega_{2},\cdots,\omega_{L}\}^{\infty}$ such that
\[
\begin{aligned} x&=\lim_{k\rightarrow \infty }f_{\omega _{l_{1}}}\circ f_{\omega _{l_{2}}}\circ \cdots \circ f_{\omega _{l_{k}}}(x_{0}) \\ &= \lim_{k\rightarrow \infty }( f_{\sigma _{1}^{l_1}}\circ \cdots \circ f_{\sigma _{j_{l_1}}^{l_1}})\circ ( f_{\sigma _{1}^{l_2}} \circ \cdots \circ f_{\sigma _{j_{l_2}}^{l_2}} ) \circ \cdots \circ ( f_{\sigma _{1}^{l_k}}\circ \cdots \circ f_{\sigma _{j_{l_k}}^{l_k}})(x_0). \end{aligned}
\]
By the uniqueness of $\sigma$ in equation (\ref{bijection-equation}), we have
$\psi(\omega)=\sigma$, showing that $\psi$ is surjective.

We are now going to show that $\mathcal{R}_{\mathcal{F}}\subseteq
\mathcal{R}_{\mathcal{G}}.$ Let $r\in\mathcal{R}_{\mathcal{F}}.$ Note that the
hyperplanes of $\mathbb{P}$ are simply the points of $\mathbb{P}$. Moreover,
the hyperplane repeller $\mathcal{R}_{\mathcal{F}}$ of $\mathcal{F}$ is simply
the attractor of the IFS $\widehat{\mathcal{F}^{-1}}:=(\widehat{\mathbb{P}%
^{n}};f_{1}^{-1},f_{2}^{-1},...,f_{M}^{-1})$ and the hyperplane repeller
$\mathcal{R}_{\mathcal{G}}$ of $\mathcal{G}$ is the attractor of
$\widehat{\mathcal{G}^{-1}}:=(\widehat{\mathbb{P}^{n}};f_{\omega_{1}}%
^{-1},f_{\omega_{2}}^{-1},\dots,f_{\omega_{L}}^{-1})$. Let $r_{0}$ be the
attractive fixed point of $f_{\omega_{1}}^{-1}.$ Note that $r_{0}$ lies in
both $\mathcal{R}_{\mathcal{G}}$ and in $\mathcal{R}_{\mathcal{F}}$. According
to Theorem \ref{main} and Theorem \ref{adjointhm}, both $\widehat{\mathcal{F}%
^{-1}}$ and $\widehat{\mathcal{G}^{-1}}$ are contractive. Therefore
\[
r=\lim_{k\rightarrow\infty}f_{\sigma_{1}}^{-1}\circ f_{\sigma_{2}}^{-1}%
\circ\cdots\circ f_{\sigma_{k}}^{-1}(r_{0})
\]
for some $\sigma=\sigma_{1}\sigma_{2}\cdots\in\{1,2,...,M\}^{\infty}.$ Since
$\psi$ is surjective, there is a string \linebreak$\omega_{l_{1}}\omega
_{l_{2}}\cdots\in\{\omega_{1},\omega_{2},\dots,\omega_{L}\}^{\infty}$ such
that
\begin{align*}
r  &  =\lim_{k\rightarrow\infty}f_{\sigma_{1}}^{-1}\circ f_{\sigma_{2}}%
^{-1}\circ\cdots\circ f_{\sigma_{k}}^{-1}(r_{0})\\
&  =\lim_{k\rightarrow\infty}\left(  f_{\sigma_{k}}^{-1}\circ f_{\sigma_{k-1}%
}\circ\cdots\circ f_{\sigma_{1}}\right)  ^{-1}(r_{0})\\
&  =\lim_{m\rightarrow\infty}\left(  f_{\omega_{l_{m}}}\circ f_{\omega
_{l_{m-1}}}\circ\cdots\circ f_{\omega_{l_{1}}}\right)  ^{-1}(r_{0})\\
&  =\lim_{k\rightarrow\infty}f_{\omega_{l_{1}}}^{-1}\circ f_{\omega_{l_{2}}%
}^{-1}\circ\cdots\circ f_{\omega_{l_{k}}}^{-1}(r_{0})\in\lim_{k\rightarrow
\infty}(\widehat{\mathcal{G}^{-1}})^{k}(r_{0})=\mathcal{R}_{\mathcal{G}}.
\end{align*}
A similar, but easier, argument shows that $\mathcal{R}_{\mathcal{G}}%
\subseteq\mathcal{R}_{\mathcal{F}}.$ Hence $\mathcal{F}$ and $\mathcal{G}$
have the same hyperplane repeller. Since the attractors and hyperplane
repellers of both are the same we have $index(\mathcal{F})=index(\mathcal{G})$
by the definition of the index.
\end{proof}

\begin{theorem}
\label{indexthm} If $\mathcal{F}=(\mathbb{P}^{1};f_{1},f_{2})$ is the
projective IFS in Proposition \ref{indexflemma} with $M=2,$ $\lambda=10,$ and
$A$ is the attractor of $\mathcal{F}$, then $index(A)=2$.
\end{theorem}

\begin{proof}
Let $\widehat{\mathcal{F}}=(\mathbb{P}^{1};\hat{f}_{1},\hat{f}_{2})$, where
\[
\hat{f}_{1}=%
\begin{pmatrix}
\frac{1}{10} & 0\\
0 & 1
\end{pmatrix}
,\qquad\qquad\hat{f}_{2}=%
\begin{pmatrix}
37 & -18\\
54 & -26
\end{pmatrix}
.
\]
It is easy to check that $\hat{f}_{1}=f\circ f_{1}\circ f^{-1}$ and $\hat
{f}_{2}=f\circ f_{2}\circ f^{-1}$ where $f_{1}$ and $f_{2}$ are the functions
in Proposition \ref{indexflemma} when $\lambda=10$, and $f$ is the projective
transformation represented by the matrix $L_{f}=%
\begin{pmatrix}
1 & -1\\
1 & -\frac{1}{2}%
\end{pmatrix}
$. It is sufficient to show that if $\widehat{A}$ is the attractor of
$\widehat{\mathcal{F}}$, then $index(\widehat{A})=2$. From here on the IFS
$\mathcal{F}$ is not used, so we drop the \textquotedblright hat" from
$\widehat{\mathcal{F}},\hat{f}_{1},\hat{f}_{2},\widehat{A}$. Also to simplify
notation, the set of points of the projective line are taken to be
$\mathbb{P=R}\cup\left\{  \mathbb{\infty}\right\}  $, where $%
\begin{pmatrix}
x\\
1
\end{pmatrix}
$ is denoted as the fraction $x$ and $%
\begin{pmatrix}
1\\
0
\end{pmatrix}
$ is denoted as $\infty$. In this notation $f_{1}(x)=\frac{1}{10}x$ and
$f_{2}(x)=\frac{37x-18}{54x-26}$ when restricted to $\mathbb{R}$. The
following are properties of $\mathcal{F}$.

\begin{enumerate}
\item The attractor $C$ of $\mathcal{F}$ is a Cantor set.

\item $index(\mathcal{F})=2$.

\item The origin $a=0$ is the attractive fixed point of $f_{1}$ while its
repulsive hyperplane is $\infty$.

\item The attractive fixed point of $f_{2}$ is at $c=2/3$ and its repulsive
hyperplane at $1/2$.

\item $C\subset\lbrack a,b]\cup\lbrack c,d]$, where $b=\frac{11}{40}-\frac
{1}{120}\sqrt{609}$ ($=0.069351$) and $d=\frac{11}{4}-\frac{1}{12}\sqrt{609}$
($=0.69351$) are the attractive fixed points of $f_{1}\circ f_{2}$ and
$f_{2}\circ f_{1}$ respectively.

\item If $h$ is any projective transformation taking $C$ into itself, then
$h([a,b]\cup\lbrack c,d])\subset\lbrack a,b]\cup\lbrack c,d]$.

\item The symmetry group of $C$ is trivial, i.e., the only projective
transformation $h$ such that $h(C)=C$ is the identity.
\end{enumerate}

Property (1) is in the proof of Proposition \ref{indexflemma}, and property
(2) is a consequence of Proposition \ref{indexflemma}. Properties (3) and (4)
are easily verified by direct calculation. Property (5) can be verified by
checking that $\mathcal{F}([a,b]\cup\lbrack c,d])\subset\lbrack a,b]\cup
\lbrack c,d]$.

To prove property (6), let $I$ denote a closed interval (on the projective
line, topologically a circle,) that contains $C.$ Its image $h^{-1}(I)$ is
also a closed interval$.$ Since $h(C)\subset C$, it follows that $C\subset
h^{-1}(C)$. Since $C$ contains $\{a,b,c,d\}$ and some points between $a$ and
$b,$ $h^{-1}(I)$ must contain $a,b$ and some points between $a$ and $b.$ It
follows that $h^{-1}(I)\supset\lbrack a,b]$. Similarly $h^{-1}(I)\supset
\lbrack c,d]$. Therefore $h^{-1}(I)\supset\lbrack a,b]\cup\lbrack c,d],$ and
hence $h([a,b]\cup\lbrack c,d])\subset I$. Now choose $I$ to be $[a,d]$ to get
(A) $h([a,b]\cup\lbrack c,d])\subset\lbrack a,d]$. Choose $I$ to be $[c,b]$
(by which we mean the line segment that goes from $c$ through $d$ then
$\infty=-\infty$ then through $a$ to end at $b$,) to obtain (B) $h([a,b]\cup
\lbrack c,d])\subset\lbrack c,b]$. It follows from (A) and (B) that
$h([a,b]\cup\lbrack c,d])\subset\lbrack a,d]\cap\lbrack c,b]=[a,b]\cup\lbrack
c,d].$

To prove property (7), assume that $h(C)=C$. We will show that $h$ must be the
identity. By property (6) $h([a,b]\cup\lbrack c,d])=[a,b]\cup\lbrack c,d]$.
Taking the complement, we have $h\left(  (b,c)\cup\left(  d,a\right)  \right)
=(b,c)\cup\left(  d,a\right)  $, and so $h([b,c]\cup\lbrack d,a])=[b,c]\cup
\lbrack d,a]$. Hence
\begin{align*}
&  h([a,b]\cup\lbrack c,d])\cap h([b,c]\cup\lbrack d,a])\\
&  =\left(  [a,b]\cup\lbrack c,d]\right)  \cap\left(  \lbrack b,c]\cup\lbrack
d,a]\right)  .
\end{align*}
It follows that $h(\{a,b,c,d\})=\{a,b,c,d\}$. Any projective transformation
that maps $\{a,b,c,d\}$ to itself must preserve the cross ratio of the four
points, so the only possibilities are (i) $h(a)=a,$ $h(b)=b,$ $h(c)=c,$
$h(d)=d,$ in which case $h$ is the identity map; (ii)
$h(a)=b,h(b)=a,h(c)=d,h(d)=c$; (iii) $h(a)=c,h(b)=d,h(c)=a,h(d)=b$, and (iv)
$h(a)=d,h(b)=c,h(c)=b,h(d)=a$. In each case one can write down the specific
projective transformation, for example, (iii) is achieved by
\[
h(x)=\frac{(d-c)(b-c)(x-a)}{(b-a+d-c)(x-c)-(d-c)(b-c)}+c\text{.}%
\]
The other two specific transformations can be deduced by permuting the symbols
$a,b,c,d.$ In each of the cases (ii), (iii) and (iv) it is straightforward to
check numerically that $h(x)$ does not map $C$ into $C$. (One compares the
union of closed intervals
\[
\lbrack f_{1}(a),f_{1}(b)]\cup\lbrack f_{1}(c),f_{1}(d)]\cup\lbrack
f_{2}(a),f_{2}(b)]\cup\lbrack f_{2}(c),f_{2}(d)],
\]
whose endpoints belong to $C$ and which contains $C$, with the union
\[
\lbrack h(f_{1}(a)),h(f_{1}(b))]\cup\lbrack h(f_{1}(c)),h(f_{1}(d))]\cup
\lbrack h(f_{2}(a)),h(f_{2}(b))]\cup\lbrack h(f_{2}(c)),h(f_{2}(d))].\text{)}%
\]
It follows that $h$ must be the identity map, as claimed. \vskip2mm

Let $\mathcal{G}=(\mathbb{P}^{1};g_{1},g_{2},....,g_{L})$ be any projective
IFS with attractor equal to $C$. The proof proceeds by showing the following:
(\ddag) for any $g\in\mathcal{G}$ we have $g=f_{\sigma_{1}}\circ f_{\sigma
_{2}}\circ...\circ f_{\sigma_{k}}$, for some $k$, where each $\sigma_{i}$ is
either $1$ or $2$. Then, by Lemma \ref{final-lemma}, $\mathcal{G}$ has the
same hyperplane repeller as $\mathcal{F}$, and hence $index(\mathcal{G}%
)=index(\mathcal{F})=2.$ So any IFS\ with attractor $C$ has index $2$. This
completes the proof of Theorem~\ref{indexthm} because it shows that any IFS
with attractor $C$ has index $2$, i.e. $index(C)=2$.

To prove claim (\ddag), consider the IFS $\mathcal{H}=([a,b]\cup\lbrack
c,d];f_{1},f_{2},g)$ where $g$ is any function in IFS $\mathcal{G}$. By
property (6) $g([a,b]\cup\lbrack c,d])\subset\lbrack a,b]\cup\lbrack c,d]$. So
$\mathcal{H}$ is indeed a well-defined IFS. It follows immediately from the
fact that both $\mathcal{F}$ and $\mathcal{G}$ have attractor equal to $C$
that $\mathcal{H}$ also has attractor $C$. It cannot be the case that
$g([a,b])\subset\lbrack a,b]$ and $g([c,d])\subset\lbrack c,d])$ since then
$g$ would have two attractive fixed points which is impossible. Similarly, it
cannot occur that $g([c,d])\subset\lbrack a,b]$ and $g([a,b])\subset\lbrack
c,d]$ for then $g^{2}$ would have two attractive fixed points, which is also
impossible. It cannot occur that $g(a)\subset\lbrack a,b]$ and $g(b)\subset
\lbrack c,d]$ for then $g([a,b]\cup\lbrack c,d])$ would not be contained in
$[a,b]\cup\lbrack c,d],$ contrary to property (6). Similarly, we rule out the
possibilities that $g(a)\subset\lbrack c,d]$ and $g(b)\subset\lbrack a,b];$
that $g(c)\subset\lbrack a,b]$ and $g(d)\subset\lbrack c,d];$ and that
$g(d)\subset\lbrack a,b]$ and $g(c)\subset\lbrack c,d]$. It follows that
either $g([a,b]\cup\lbrack c,d])\subset\lbrack a,b]\subset f_{1}([a,d]),$ or
$g([a,b]\cup\lbrack c,d])\subset\lbrack c,d]\subset f_{2}([c,b])$ where
$[c,b]$ denotes the interval from $c$ to $\infty$ then from $-\infty$ to $b.$
(Here, the containments $[a,b]$ $\subset f_{2}([c,b])$ and $[c,d]\subset
f_{2}([c,b])$ are readily verified by direct calculation.) It now follows that
either $g(C)\subset C\cap f_{1}([a,d])=f_{1}(C)$ or $g(C)\subset C\cap
f_{2}([c,b])=f_{2}(C).$ Hence
\[
g(C)\subset C_{\sigma_{1}}:=f_{\sigma_{1}}(C)
\]
for $\sigma_{1}\in\{1,2\}$. If $g(C)=C_{\sigma_{1}}$ then $h(C)=C$ where $h$
is the projective transformation $f_{\sigma_{1}}^{-1}\circ g$. In this case
property (7) implies that $h$ must be the identity map. Therefore
\[
g=f_{\sigma_{1}}\text{.}%
\]
If, on the other hand, $g(C)\varsubsetneq f_{\sigma_{1}}(C)$ then we consider
the IFS
\[
\mathcal{H}_{\sigma_{1}}=(f_{\sigma_{1}}([a,b]\cup\lbrack c,d]);f_{\sigma_{1}%
}\circ f_{1}\circ f_{\sigma_{1}}^{-1},f_{\sigma_{1}}\circ f_{2}\circ
f_{\sigma_{1}}^{-1},g\circ f_{\sigma_{1}}^{-1})\text{,}%
\]
It is readily checked that the functions that comprise this IFS indeed map
$f_{\sigma_{1}}([a,b]\cup\lbrack c,d])$ into itself. The attractor of
$\mathcal{H}_{\sigma_{1}}$ is $C_{\sigma_{1}}=f_{\sigma_{1}}(C)$ because
\begin{align*}
\mathcal{H}_{\sigma_{1}}(C_{\sigma_{1}})  &  =f_{\sigma_{1}}\circ f_{1}\circ
f_{\sigma_{1}}^{-1}\left(  f_{\sigma_{1}}(C)\right)  \cup f_{\sigma_{1}}\circ
f_{2}\circ f_{\sigma_{1}}^{-1}\left(  f_{\sigma_{1}}(C)\right)  \cup g\circ
f_{\sigma_{1}}^{-1}\left(  f_{\sigma_{1}}(C)\right) \\
&  =f_{\sigma_{1}}(f_{1}(C)\cup f_{2}(C))\cup g(C)=f_{\sigma_{1}}(C)\cup
g(C)\\
&  =C_{\sigma_{1}}\text{ (because }g(C)\subset f_{\sigma_{1}}(C)\text{).}%
\end{align*}
Let
\[
a_{\sigma_{1}}<b_{\sigma_{1}}<c_{\sigma_{1}}<d_{\sigma_{1}}%
\]
denote the endpoints of the two intervals $f_{\sigma_{1}}([a,b])$ and
$f_{\sigma_{1}}([c,d])$, and write our new IFS as
\[
\mathcal{H}_{\sigma_{1}}=([a_{\sigma_{1}},b_{\sigma_{1}}]\cup\lbrack
c_{\sigma_{1}},d_{\sigma_{1}}];f_{\left(  \sigma_{1}\right)  1},f_{\left(
\sigma_{1}\right)  2},g_{\sigma_{1}}),
\]
where
\[
f_{\left(  \sigma_{1}\right)  \sigma_{2}}=f_{\sigma_{1}}\circ f_{\sigma_{2}%
}\circ f_{\sigma_{1}}^{-1},\text{ and }g_{\sigma_{1}}=g\circ f_{\sigma_{1}%
}^{-1}.
\]
Repeat our earlier argument to obtain
\[
g_{\sigma_{1}}([a_{\sigma_{1}},b_{\sigma_{1}}]\cup\lbrack c_{\sigma_{1}%
},d_{\sigma_{1}}])\subset f_{\sigma_{2}}([a_{\sigma_{1}},b_{\sigma_{1}}%
]\cup\lbrack c_{\sigma_{1}},d_{\sigma_{1}}]),
\]
and in particular that
\[
g_{\sigma_{1}}(C_{\sigma_{1}})\subset C_{\sigma_{1}\sigma_{2}}:=f_{\left(
\sigma_{1}\right)  \sigma_{2}}(C_{\sigma_{1}})=f_{\sigma_{1}}\circ
f_{\sigma_{2}}\circ f_{\sigma_{1}}^{-1}\circ f_{\sigma_{1}}(C)=f_{\sigma_{1}%
}\circ f_{\sigma_{2}}(C)
\]
for some $\sigma_{2}\in\{1,2\}.$ If $g_{\sigma_{1}}(C_{\sigma_{1}}%
)=C_{\sigma_{1}\sigma_{2}}$ then $g_{\sigma_{1}}(f_{\sigma_{1}}(C))=f_{\sigma
_{1}}\circ f_{\sigma_{2}}(C)$ which implies $g\circ f_{\sigma_{1}}^{-1}\circ
f_{\sigma_{1}}(C)=f_{\sigma_{1}}\circ f_{\sigma_{2}}(C)$ which implies, as
above, that
\[
g=f_{\sigma_{1}}\circ f_{\sigma_{2}}.
\]
If $g_{\sigma_{1}}(C_{\sigma_{1}})\varsubsetneq C_{\sigma_{1}\sigma_{2}}$ then
we construct a new projective IFS $H_{\sigma_{1}\sigma_{2}}$ in the obvious
way and continue the argument. If the process does not terminate with
\[
g=f_{\sigma_{1}}\circ f_{\sigma_{2}}\circ\cdots\circ f_{\sigma_{k}}%
\]
for some $k$, then $g(C)$ is a singleton, which is impossible because $g$ is
invertible. We conclude that
\[
\mathcal{G}=(\mathbb{P};f_{\omega_{1}},f_{\omega_{2}},\dots,f_{\omega_{L}})
\]
where
\[
f_{\omega_{l}}=f_{\sigma_{1}^{l}}\circ f_{\sigma_{2}^{l}}\circ\cdots\circ
f_{\sigma_{k_{l}}^{l}}%
\]
in the obvious notation. This concludes the proof of claim (\ddag).

Now Lemma \ref{final-lemma} implies $index(\mathcal{G})=index(\mathcal{F})$.
So the index of any projective IFS that has $C$ as its attractor is $2.$ It
follows that $index(A)=2.$
\end{proof}

\section{\label{remarksec}\textbf{Remarks}}

Various remarks are placed in this section so as to avoid interrupting the
flow of the main development.

\begin{remark}
Example 3 in Section \ref{examplesec} illustrates that there exist
non-contractive projective IFSs that, nevertheless, have attractors. Such IFSs
are not well understood and invite further research.
\end{remark}

\begin{remark}
\label{remark1}It is well known \cite{williams} that if each function of an
IFS is a contraction on a complete metric space $X$, then $\mathcal{F}$ has a
unique attractor in $X$. So statement (4) of the Theorem~\ref{main}
immediately implies the existence of an attractor $A$, but not that there is a
hyperplane $H$ such that $A\cap H=\emptyset$.
\end{remark}

\begin{remark}
\label{remark3} Let $\mathcal{F}$ be a contractive IFS. By Corollary~\ref{cor}%
, each $f\in\mathcal{F}$ has an invariant hyperplane $H_{f}$. If all these
invariant hyperplanes are identical, say $H_{f}=H$ for all $f\in\mathcal{F}$,
then the projective IFS $\mathcal{F}$ is equivalent to an affine IFS acting on
the embedded affine space $\mathbb{P}^{n}\smallsetminus H$. More specifically,
let $G=(\mathbb{R}^{n};g_{1},g_{2},...,g_{M})$ be an affine IFS where
$g_{i}(x)=L_{i}^{\prime}(x)+t_{i}$ and where $L_{i}^{\prime}$ is the linear
part and $t_{i}$ the translational part. A corresponding projective IFS is
$\mathcal{F}=(\mathbb{P}^{n};f_{1},f_{2},...,f_{M})$ where, for each $i$ the
projective transformation $f_{i}$ is represented by the matrix $L_{f_{i}}$:
\[
L_{f_{i}}
\begin{pmatrix}
x_{0}\\
x_{1}\\
.\\
x_{n}%
\end{pmatrix}
=%
\begin{pmatrix}
1 & 0\\
t_{i} & L_{i}^{^{\prime}}%
\end{pmatrix}%
\begin{pmatrix}
x_{0}\\
x_{1}\\
.\\
x_{n}%
\end{pmatrix}
,
\]
Here $\mathbb{R}^{n}$ corresponds to $\mathbb{P}\backslash H$ with $H$ the
hyperplane $x_{0}=0$. In this case the hyperplane repeller of $\mathcal{F}$ is
$H.$
\end{remark}

\begin{remark}
\label{remark4} Straightforward geometrical comparisons between $d_{K}(x,y)$
and $d_{\mathbb{P}}(x,y)$ show that (i) the two metrics are bi-Lipshitz
equivalent on any convex body contained in $int\left(  K\right)  $ and (ii) if
$f$ is any projective transformation on $\mathbb{P}^{n}$ then the metric
$d_{f(\mathbb{P)}}$ defined by $d_{f(\mathbb{P)}}(x,y)=d_{\mathbb{P}%
}(f(x),f(y))$ for all $x,y\in\mathbb{P}^{n}$ is bi-Lipschitz equivalent to
$d_{\mathbb{P}}$. A consequence of assertions (i) and (ii) is that the value
of the Hausdorff dimension of any compact subset of $int\left(  K\right)  $ is
the same if it is computed using the round metric $d_{\mathbb{P}}$ or the
Hilbert metric $d_{K}$; see \cite[Corollary 2.4, p.30]{falconer}, and its
value is invariant under the group of projective transformations on
$\mathbb{P}^{n}.$ In particular, the Hausdorff dimension of an attractor of a
projective IFS is a projective invariant.
\end{remark}

\begin{remark}
Theorem 1 provides conditions for the existence of a metric with respect to
which a projective IFS is contractive. In so doing, it invites other
directions of development, including IFS with place-dependent probabilities
\cite{comptes}, graph-directed IFS theory \cite{mauldin}, projective fractal
interpolation, and so on. In subsequent papers we hope to describe a natural
generalization of the joint spectral radius and applications to digital imaging.
\end{remark}

\begin{remark}
\label{attractremark} Definition \ref{attractdef} of the attractor of an IFS
is a natural generalization of the definition \cite[p.82]{FE1} of the
attractor of a contractive IFS. Another general definition, in the context of
iterated closed relations on a compact Hausdorff space, has been given by
McGehee \cite{mcgeehee}. He proves that his definition is equivalent to
Definition \ref{attractdef}, for the case of contractive iterated function
systems. However, readily constructed examples show that McGehee's definition
of attractor is weaker than Definition \ref{attractdef}.
\end{remark}

\textbf{ACKNOWLEDGMENTS. }We thank David C. Wilson for many helpful
discussions; he influenced the style and contents of this paper and drew our
attention to the Krein-Rutman theorem.

\end{document}